\documentclass[12pt]{amsart}
\usepackage[hmargin=0.8in,height=8.6in]{geometry}
\usepackage{amssymb,amsthm, times}
\usepackage{delarray,verbatim}
\usepackage{ifpdf}
\ifpdf
\usepackage[pdftex]{graphicx}
 \else
\usepackage[dvips]{graphicx}
 \fi

\usepackage{bm}

\usepackage{ifpdf}
\usepackage{color}
\definecolor{webgreen}{rgb}{0,.5,0}
\definecolor{webbrown}{rgb}{.8,0,0}
\definecolor{emphcolor}{rgb}{0.95,0.95,0.95}

\usepackage{hyperref}
\hypersetup{%
          colorlinks=true,
          linkcolor=webbrown,
          filecolor=webbrown,
          citecolor=webgreen,
          breaklinks=true}
\ifpdf \hypersetup{pdftex,
            pdfstartview=FitH, %%Fit, FitB, FitH
            bookmarksopen=true,
            bookmarksnumbered=true
} \else \hypersetup{dvips} \fi

\newcommand {\ud}{{\rm d}}

\newcommand {\B}{\mathcal{B}}

\numberwithin{equation}{section}

\newtheorem{theorem}{Theorem}[section]

\newtheorem{remark}{Remark}[section]
\newtheorem{lemma}{Lemma}[section]

\numberwithin{remark}{section} \numberwithin{proposition}{section}
\numberwithin{corollary}{section}
\newcommand {\R}{\mathbb{R}}

\newcommand {\p}{\mathbb{P}}

\newcommand {\E}{\mathbb{E}}

\newcommand{\diff}{{\rm d}}

\newcommand{\lev}{L\'{e}vy }

\newcommand{\e}{\mathbb{E}}

\begin{document}
\title[Periodic barrier strategies for a spectrally positive L\'evy process]{On the optimality of periodic barrier strategies for a spectrally positive L\'evy process}

\thanks{This version: \today.  }
\author[J. L. P\'erez]{Jos\'e-Luis P\'erez$^*$}
%\address[J. L. P\'erez]{.}
%\email{ }
\thanks{$*$\, Department of Probability and Statistics, Centro de Investigaci\'on en Matem\'aticas A.C. Calle Jalisco s/n. C.P. 36240, Guanajuato, Mexico. Email: jluis.garmendia@cimat.mx.  }
\author[K. Yamazaki]{Kazutoshi Yamazaki$^\dag$}
%\address[K. Yamazaki]{}
\thanks{$\dag$\, (corresponding author) Department of Mathematics,
Faculty of Engineering Science, Kansai University, 3-3-35 Yamate-cho, Suita-shi, Osaka 564-8680, Japan. Email: kyamazak@kansai-u.ac.jp.   }
%\email{}
\date{}

\maketitle

\begin{abstract}
We study the optimal dividend problem in the dual model where dividend payments can only be made at the jump times of an independent Poisson process. In this context, Avanzi et al.\ \cite{ATW} solved the case with i.i.d.\ hyperexponential jumps; they showed the optimality of a (periodic) barrier strategy where dividends are paid at dividend-decision times if and only if the surplus is above some level. In this paper, we generalize the results for a general spectrally positive \lev process with additional terminal payoff/penalty at ruin, and also solve the case with classical bail-outs so that the surplus is restricted to be nonnegative. The optimal strategies as well as the value functions are concisely written in terms of the scale function. Numerical results are also given.
\\
\noindent \small{\noindent  AMS 2010 Subject Classifications: 60G51, 93E20, 91B30 \\ %C44, C61, G24, G32, G35 \\
}\\
\end{abstract}

%
%\red{JL: I am wondering if it is clean to set sections like this
%\begin{enumerate}
%\item introduction
%\item formulation of the problems (both of them)
%\item scale functions, review of the Parisian reflected processes and the fluctuation identities used in this paper.
%\item solution to the first problem.
%\item solution to the second problem.
%\item numerical results (to be added later)
%\end{enumerate}
%Also, about notations, if you like, we can use commonly used ones
%\begin{enumerate}
%\item $W_q \rightarrow W^{(q)}$
%\item $\Phi_q \rightarrow \Phi(q)$ 
%\end{enumerate}
%Then, it may be natural to use $Z^{(q,r)} \rightarrow Z^{(q,r)}$. 
%}
%
%\red{
%For the processes, it may be cleaner to use $(t)$ because, this time we need to add $b$ (then two things on the superscript) and so it may be an option to write $X(t)$, $L_r^b(t)$, $R_r^b(t)$
%}

%\red{I guess the name is a little long.  How about something like ``On the optimality of periodic barrier strategies for a spectrally positive \lev process" or something?}

%\red{Also, we are using many names for the reflection strategy in the paper.  Avanzi et al uses ``periodic barrier strategy" and so I guess we should use it as well? For the second problem, ``periodic barrier strategy with classical reflection below? any suggestions? }

\section{Introduction}

In risk theory,  the model of periodic payments has drawn much attention recently. 
While a majority of the existing continuous-time models assume that dividends can be made at all times and instantaneously, in reality dividend decisions can only be made at some intervals.  Solving the optimal dividend problem under periodic payments is in general difficult.  However, thanks to the recent developments of the fluctuation theory, in particular, of \lev processes, it is getting more tractable.
%  if one assumes that the interarrival times of payment opportunities are  exponentially distributed, one can expect a more tractable solution.

%It gives a more realistic model than the classical setting where dividend payments are assumed to be made at all times and instantaneously. In reality, dividend decisions can only be made at some intervals as opposed to the assumption of being made continuously.  Typically the interarrival times of the payment opportunities are assumed to be  exponentially distributed in order to obtain closed form expressions.

In this paper, we consider the optimal dividend problem under the constraint that dividend payments can only be made at the jump times of an independent Poisson process.  We focus on the dual model (or the spectrally positive \lev model), which is known to be an appropriate model for a company driven by inventions or discoveries (see, e.g.\ \cite{AGS2007, AGS2008, ASW2011, BKY, BKY2, LLZ2016, MP2016, YSP2, ZWYC2016}). In this context, Avanzi et al.\ \cite{ATW} solved the case with i.i.d.\ hyperexponential jumps. Our objective is to generalize their results for a general spectrally positive \lev process with a terminal payoff (penalty) at ruin, and also solve its extension with classical bail-outs so that the surplus is restricted to be nonnegative uniformly in time.   %[added this to reference this journal paper by Irmina and Palmowski.]}
Recently, Zhao et al.\ \cite{ZCY} studied similar problems where they consider the case with no terminal payoff (penalty) at ruin but with fixed cost for capital injection. For a related problem with Parisian delay, see, among others, \cite{CzarnaPalmowski}.

In order to solve the problem, we use the recent results given in Avram et al.\  \cite{APY}.  As has been already confirmed in \cite{ATW}, the \emph{periodic barrier strategy} is expected to be optimal. Namely, at each dividend-decision time, dividends are paid if and only if the surplus is above some barrier and then it is pushed down to the barrier. The resulting surplus process becomes the dual of the \emph{Parisian-reflected process} considered in  \cite{APY}.  Therefore the expected net present value (NPV) of dividends can be computed concisely using the scale function, which enables one to follow the classical ``\emph{guess and verify}'' technique described below:
\begin{enumerate}
\item In the guessing step, the candidate barrier level $b^*$ is first chosen. Proceeding like in the existing literature (see, e.g., \cite{APP2007, BKY, BKY2, BLS, HPY}), $b^*$ (if strictly positive) is set so that the value function becomes ``smooth'' at the barrier.  
Differently from the classical dual model as in \cite{BKY} where the value function becomes $C^1 (0, \infty)$ (resp.\ $C^2 (0, \infty)$) for the case $X$ is of bounded (resp.\ unbounded) variation (see \cite{BKY2} for the case there is a fixed cost), we shall see in the periodic payment case that  the value function becomes $C^2 (0, \infty)$ (resp.\ $C^3 (0, \infty)$) for the case $X$ is of bounded (resp.\ unbounded) variation.  
\item In the verification step, we first obtain the verification lemma, or sufficient conditions for optimality, and then show that the candidate value function corresponding to the selected periodic barrier strategy satisfies all the conditions.  We shall see that its slope is larger (resp.\ smaller) than $1$ at the position below (resp.\ above)  the barrier.  This together with the martingales constructed using scale functions completes the proof. %{\color{blue} Kazu: I am not sure if martingale properties of the scale function is adequate. How about "some martingales constructed using scale functions" o something similar?]}
\end{enumerate}
We see that $b^* = 0$ can be possible and in this case the \emph{taking all the money and run strategy at the first opportunity} becomes optimal. %\red{how about ``the taking all the money and run strategy becomes optimal?"}.  
As has been observed in \cite{ATW}, this can happen even when (the terminal payoff is zero and) the underlying \lev process drifts to infinity, while in the classical model this happens if and only if the process drifts to $-\infty$ or oscillates.

In our second problem, we consider the case with classical bail-outs, where capital must be injected so that the surplus process remains nonnegative uniformly in time; see \cite{ASW2011, BKY} for the classical case. The objective is to maximize the expected NPV of dividends minus the costs of capital injection.  Using the results in \cite{APY}, the expected NPV under the periodic barrier strategy can be computed.   Again, we select the candidate barrier $b^\dagger$ using the same smoothness conditions described above.  The optimality is shown similarly by the verification arguments.  In fact, most of the results hold verbatim because  the resulting value function admits \emph{the same form} as that for the first problem, except that the barrier level is different.

In both problems, the optimal barrier and the value function can be written concisely using the scale function. In order to confirm the obtained analytical results, we give a sequence of numerical experiments using the phase-type \lev process that admits an analytical form of scale function, and hence the solutions can be instantaneously computed.  We shall confirm the optimality and also analyze the behaviors as the frequency of dividend-decision opportunities increases.

Before closing the introduction, we discuss here the connections with the results in Zhao et al.\ \cite{ZCY}.  The first problem considered in \cite{ZCY}  is the special case of our first problem with no  terminal  payoff/cost at ruin. While our paper directly uses the results of Avram et al.\ \cite{APY} to derive the expected NPV of dividends under the periodic barrier strategy, they obtained it in a different way using the results by Albrecher et al.\ \cite{AIZ}, which gives the identities for spectrally negative \lev processes observed at Poisson arrival times.  For the selection of optimal barrier and verification of optimality, several results  in the current paper (Lemmas \ref{lemma_opt_thres} and \ref{cond_b^*}, in particular) are used.  The second problem in \cite{ZCY} is a variant of our second problem (with capital injection) where they consider the case with a fixed cost for capital injection.  With the existence of a fixed cost, the set of capital injection strategies is restricted to be a set of impulse control.  As shown in \cite{ZCY}, their value function converges, as the fixed cost decreases to zero, to that of our second problem.

The rest of the paper is organized as follows. In Section \ref{section_preliminaries}, we review the spectrally positive \lev process and define the two problems to be considered in this paper.   In Section  \ref{PR}, we define the periodic barrier strategy (with and without the classical reflection below) and construct the corresponding surplus process.  We review the scale function and give the expected NPVs  corresponding to these strategies.  Sections \ref{section_first_prob} and \ref{section_second_prob} solve the first and second problems, respectively.  
%We conclude the paper with numerical results in Section \ref{section_numerics}. 
Section \ref{section_numerics} gives numerical results and Section \ref{section_conclusion} concludes the paper. %\blue{concluding remarks?}. \red{[How about ``concludes the paper"?]}
The proofs of the verification lemmas are deferred to the appendix.

Throughout the paper,  $x+ := \lim_{y \downarrow x}$ and $x-  := \lim_{y \uparrow x}$ are used to indicate the right- and left-hand limits, respectively. 
We let 
$\Delta \zeta(s):= \zeta(s)-\zeta(s-)$ and $\Delta w(\zeta(s)):=w(\zeta(s))-w(\zeta(s-))$ for any process $\zeta$ with left-limits.

	\section{Preliminaries} \label{section_preliminaries}
\subsection{Spectrally positive L\'evy processes} Let $X=(X(t); t\geq 0)$ be a L\'evy process defined on a  probability space $(\Omega, \mathcal{F}, \p)$.  For $x\in \R$, we denote by $\p_x$ the law of $X$ when it starts at $x$ and write for convenience  $\p$ in place of $\p_0$. Accordingly, we shall write $\e_x$ and $\e$ for the associated expectation operators. In this paper, we shall assume throughout that $X$ is \textit{spectrally positive},   meaning here that it has no negative jumps and that it is not a subordinator.  We will assume throughout this work that its Laplace exponent $\psi:[0,\infty) \to \R$, i.e.
	\[
	\e\big[{\rm e}^{-\theta X(t)}\big]=:{\rm e}^{\psi(\theta)t}, \qquad t, \theta\ge 0,
	\] %\red{Do you like to remove the cut off function?}{\color{blue}[Do you mind if we leave it Kazu, because if we change it we may have to modify the generator $\mathcal{L}$ and the proofs in the Appendix?]} \red{ok. I just moved the conditon later when we define the problem.}
	is given, by the \emph{L\'evy-Khintchine formula}
	\begin{equation}\label{lk}
		\psi(\theta):=\gamma\theta+\frac{\sigma^2}{2}\theta^2+\int_{(0, \infty)}\big({\rm e}^{-\theta z}-1+\theta z \mathbf{1}_{\{z<1\}}\big)\Pi(\ud z), \quad \theta \geq 0,
	\end{equation}
	where $\gamma \in \R$, $\sigma\ge 0$, and $\Pi$ is a measure on $(0, \infty)$ called the L\'evy measure of $X$ that satisfies
	\[
	\int_{(0, \infty)}(1\land z^2)\Pi(\ud z)<\infty.
	\]

%\red{JL: I think we need to assume $\E X_1 < \infty$ and so we can write the Laplace exponent
%	\begin{equation*}
%		\psi(\theta):=\gamma\theta+\frac{\sigma^2}{2}\theta^2+\int_{(0, \infty)}\big({\rm e}^{-\theta x}-1+\theta x \big)\Pi(\ud x), \quad \theta \geq 0,
%	\end{equation*}	instead?
%}

	It is well-known that $X$ has paths of bounded variation if and only if $\sigma=0$ and $\int_{(0,1)} z\Pi(\mathrm{d}z) < \infty$; in this case, $X$ can be written as
	\begin{equation}
		X(t)=-ct+S(t), \,\,\qquad t\geq 0,\notag
		%\label{BVSNLP}
	\end{equation}
	where 
	\begin{align}
		c:=\gamma+\int_{(0,1)} z\Pi(\mathrm{d}z) \label{def_drift_finite_var}
	\end{align}
	and $(S(t); t\geq0)$ is a driftless subordinator. Note that  necessarily $c>0$, since we have ruled out the case that $X$ has monotone paths; its Laplace exponent is given by
	\begin{equation*}
		\psi(\theta) = c \theta+\int_{(0,\infty)}\big( {\rm e}^{-\theta z}-1\big)\Pi(\ud z), \quad \theta \geq 0.
	\end{equation*}
For the rest of the paper, we assume that
	\begin{align}
	\E [X(1)] = - \psi'(0+) < \infty,
	\end{align}
	so that the problem considered below will have nontrivial solutions.

%\subsection{Formulation of the problems.}{\color{blue} Working title.}
\subsection{The optimal dividend problem with Poissonian dividend-decision times and terminal payoff/penalty at ruin.}\label{dividends-strategy} %\red{JL: How about something like ``Optimal dividend problem with Poissonian dividend decision times"?}

In our first problem, we will assume that the dividend payments can only be made at the arrival times of a Poisson process $N^r=( N^r(t); t\geq 0) $ with intensity $r>0$, which is independent of the L\'evy process $X$.  %\red{[moved this sentence here]}
The set of dividend-decision times is denoted by $\mathcal{T}_r :=(T(i); i\geq 1 )$, where $T(i)$, for each $i\geq 1$, represents the $i^{\textrm{th}}$ arrival time of the Poisson process $N^r$. This implies that $T(i)-T(i-1)$, $i\geq 1$ (with $T(0) := 0$) are exponentially distributed with mean $1/r$.  Let $\mathbb{F} := (\mathcal{F}(t); t \geq 0)$ be the filtration generated by the process $(X, N^r)$.

\par In this setting, a strategy  $\pi := \left( L^{\pi}(t); t \geq 0 \right)$ is a nondecreasing, right-continuous, and $\mathbb{F}$-adapted process where  the cumulative amount of dividends $L^{\pi}$ admits the form
\[
L^{\pi}(t)=\int_{[0,t]}\nu^{\pi}(s)\diff N^r(s),\qquad\text{$t\geq0$,}
\]
for some $\mathbb{F}$-adapted c\`agl\`ad process $\nu^{\pi}$.
%Given that we can \red{pay dividends only} at the arrival times of the process $N^r$, we have that the process $L^{\pi}$ takes the following form
Here, for each $t\geq0$, $\nu^\pi(t)$ represents the dividend payment at time $t$ associated with the strategy $\pi$. %In particular, the dividend payment at time $T(i)$  is given by $\nu^\pi(T(i))$ for each $i\geq 1$. In other words, a strategy $\pi$ can also be defined by a set $\mathcal{V}^{\pi}:=\{\nu^{\pi}(T(i)):i\geq 1\}$.\red{[lets remove the last sentence.]}
 % represents one particular dividend strategy. 

The surplus process $U^\pi$ after dividends are deducted is such that %$U^\pi(0-) := x$ \red{[I guess we can remove $U^\pi(0-) := x$? because control at time $0$ is not possible? (bail-out case would be different when $x < 0$ though)]} and 
\[
U^\pi(t) := X(t)- L^\pi (t) =X(t)-\sum_{i=1}^{\infty}\nu^{\pi}(T(i))1_{\{T(i)\leq t\}}, \qquad\text{$0\leq t\leq \tau_0^{\pi}$}, 
\]
where 
\begin{align*}
\tau_0^{\pi}:=\inf\{t>0:U^{\pi}(t)<0\}
\end{align*}
 is the corresponding ruin time.  Here and throughout, let $\inf \varnothing = \infty$. While the payment of dividends is allowed to cause immediate ruin, it cannot exceed the amount of surplus currently available. In other words, we also assume that
\begin{align} \label{surplus_constraint}
0\leq  \Delta L^{\pi}(T(i))=\nu^{\pi}(T(i))\leq U^{\pi}(T(i)-), \qquad\text{for $i \geq 1$.}
\end{align}
%\red{[we can define $\Delta$ here?  It is currently defined in the appendix.]}
%\red{[moved here]
%where we use the following notation: $\Delta L^{\pi}(T(i)):=L^{\pi}(T(i))-L^{\pi}(T(i)-)$.
Let $\mathcal{A}$ be the set of all admissible strategies  that satisfy all the constraints described above.

%\par Then we can express the surplus process $U^{\pi}$ in the following form
%\[
%U^{\pi}(t)=X(t)-\sum_{i=1}^{\infty}\nu^{\pi}(T(i))1_{\{T(i)\leq t\}}, \qquad\text{$0\leq t\leq \tau_0^{\pi}$}.
%\]
The problem is to maximize, for $q > 0$, the expected NPV of dividends paid until ruin and the terminal payoff at ruin $\rho \in \R$ (penalty if it is negative)  associated with the strategy $\pi\in \mathcal{A}$, defined as
\begin{align*}
	v_{\pi} (x) := \mathbb{E}_x \Big( \int_{[0,\tau_0^{\pi}]} {\rm e}^{-q t} \diff L^{\pi}(t) + \rho {\rm e}^{-q \tau_0^\pi}\Big) = \mathbb{E}_x \Big(\int_{[0,\tau_0^{\pi}]} {\rm e}^{-q t}  \nu^{\pi}(t)\diff N^r(t) + \rho {\rm e}^{-q \tau_0^\pi}\ \Big), \quad x \geq 0.
\end{align*}
Hence the problem is to compute the value function
\begin{equation*}
	v(x):=\sup_{\pi \in \mathcal{A}}v_{\pi}(x), \quad x \geq 0,
\end{equation*}
and obtain the optimal strategy $\pi^*$ that attains it, if such a strategy exists.
%\red{JL: $v_\pi$ or $v^\pi$?}

%\red{[move this paragraph later and combine with current Section 2?]}
\subsection{Extension with classical bail-outs.} \label{second_prob}%\red{Here, how about ``Extension with classical bail-outs"?}
In our second problem,
%we are interested in a variation of the problem we discussed in the previous section.  We 
we consider a version  where the time horizon is infinity, and the shareholders are required to inject capital to prevent the company from going bankrupt, with extra conditions on the dividend strategy described below.

%\red{JL: To spell out the difference from the above, how about using $\bar{\pi} := \left( L^{\bar{\pi}}(t), R^{\bar{\pi}}(t); t \geq 0 \right)$?}

A strategy  is a pair $\bar{\pi} := \left( L^{\bar{\pi}}(t), R^{\bar{\pi}}(t); t \geq 0 \right)$ of nondecreasing, right-continuous, and $\mathbb{F}$-adapted processes where $L^{\bar{\pi}}$ is the cumulative amount of dividends and $R^{\bar{\pi}}$ is that of injected capital. The corresponding risk process is given by $U^{\bar{\pi}}(0-) := x$ and 
\begin{align*}
U^{\bar{\pi}}(t) := X(t) - L^{\bar{\pi}}(t) + R^{\bar{\pi}}(t), \quad t \geq 0,
\end{align*}
and $(L^{\bar{\pi}}, R^{\bar{\pi}})$ must be chosen so that $U^{\bar{\pi}}$ stays nonnegative uniformly in time.
% \par In addition,  we will assume that the cumulative amount of dividends can only occur at the arrival times of a Poisson process, and hence it has the same form described in Section \ref{dividends-strategy}. 
 %\red{JL: This condition probably is not necessarily for this problem because if you do it, it is suboptimal as $\beta > 1$.}

\par In addition,  we will assume that the cumulative amount of dividends can only occur at the arrival times of a Poisson process in $\mathcal{T}_r$, and so, in a similar way as in Section \ref{dividends-strategy}, we have that $L^{\bar{\pi}}$ admits the form
\[
L^{\bar{\pi}}(t)=\int_{[0,t]}\nu^{\bar{\pi}}(s) \diff N^r(s),\qquad\text{$t \geq 0$,}
\]
for some $\mathbb{F}$-adapted c\`agl\`ad process $\bar{\nu}^{\pi}$.
%$\nu^{\bar{\pi}}(t)$ represents the dividend payment at time $t$ associated with the strategy $\bar{\pi}$.
%\red{[let's remove ``where ..."]}

Assuming that $\beta > 1$ is the cost per unit injected capital and $q > 0$ is the discount factor, we want to maximize %\red{change to $u_{\bar{\pi}}$ or $w_{\bar{\pi}}$?}
\begin{align*}
	u_{\bar\pi} (x) := \mathbb{E}_x \left( \int_{[0, \infty)} {\rm e}^{-q t} \diff L^{\bar{\pi}}(t) - \beta \int_{[0, \infty)} {\rm e}^{-q t} \diff R^{\bar{\pi}}(t)\right), \quad x \geq 0,
\end{align*}
%\red{[moved back here again because we need $q$]
over  the set of all admissible strategies $\bar{\mathcal{A}}$ that satisfy all the constraints described above and 
\begin{align}
	\E_x\left(\int_{[0, \infty)} {\rm e}^{-qt} \diff R^{\bar{\pi}}(t)\right) < \infty. %\quad a.s.
	 \label{admissibility2}
\end{align}
Hence the problem is to compute the value function
\begin{equation*}
	u(x):=\sup_{\bar{\pi} \in \bar{\mathcal{A}}}u_{\bar{\pi}}(x), \quad x \geq 0,
\end{equation*}
and obtain an optimal strategy $\bar{\pi}^*$ that attains it, if such a strategy exists.
%\red{JL: How about $\bar{\mathcal{A}}$?}

%\red{[similarly we move later?]}
	\section{%L\'evy processes Reflected at Poissonian times
	Periodic barrier strategies}\label{PR}
	
%\red{[JL: Moved this paragraph here]} 
Our objective for the first problem is to show the optimality of the periodic barrier strategy%reflection strategy at Poissonian times
, say $\pi^{b}$, with a suitable barrier level $b \geq 0$. Namely, at each Poissonian dividend-decision time, dividends are paid whenever the surplus process is above $b$ and is pushed down so that the remaining surplus becomes $b$. The controlled process, which we formally construct below, is precisely the dual process of the \emph{Parisian-reflected process} considered in \cite{APY}.

	%\red{JL: I think we should move this section after the model and move it to the current Section 4? Then, we don't have to define $\mathcal{T}_r=\{T(i); i \geq 1\}$ twice. }
	
	%\red{JL: $T(i)$ or $T(i)$? Probably better to write $T(i)$ because otherwise I don't know how to write$T_0^-(i)$?}
	
%\red{[Combine Sections 1 and 2?]} 
With $\mathcal{T}_r=(T(i); i \geq 1)$, the set of jump times of an independent Poisson process defined in Section \ref{section_preliminaries}, we construct the \emph{\lev process with Parisian reflection above} $U_r^{b} = (U_r^{b}(t); t \geq 0)$ as follows: the process is only observed at times $\mathcal{T}_r$ and is pushed down to $b$ if only if it is above $b$.
	
	More specifically, we have
	\begin{align} \label{X_X_r_the_same}
		U_r^{b}(t) = X(t), \quad 0 \leq t < T_b^+(1)
	\end{align}
	where
	\begin{align} T_b^+(1) := \inf\{T(i):\; X(T(i)) >b\}. \label{def_T_0_1}
	\end{align}
	The process then jumps downward by $X(T_b^+(1))-b$ so that $U_r^{b}(T_b^+(1)) = b$. For $T_b^+(1) \leq t < T_b^+(2)  := \inf\{T(i) > T_b^+(1):\; U^{b}_r(T(i)-) > b\}$, we have $U_r^{b}(t) = X(t) - (X(T_b^+(1))-b)$.  The process $U_r^{b}$ can be constructed by repeating this procedure.
	
	%\red{JL: Similarly to $U_r^b$, how about writing $L_r^b$?}
	
	Suppose $L_r^b(t)$ is the cumulative amount of (Parisian) reflection until time $t \geq 0$. Then we have
	\begin{align*}
		U_r^{b}(t) = X(t) - L_r^b(t), \quad t \geq 0,
	\end{align*}
	%	\red{JL: Maybe better to use $U_r^{b}$ instead of $X^{r,b}$ because $U$ is the controlled process?}
	with
	\begin{align}
		L_r^b(t) := \sum_{T_b^+(i) \leq t} \left(U_r^{b}(T_b^+(i)-)-b\right), \quad t \geq 0, \label{def_L_r}
	\end{align}
	where $(T_b^+(n); n \geq 1)$ can be constructed inductively by \eqref{def_T_0_1} and
	\begin{eqnarray*}T_b^+(n+1) := \inf\{T(i) > T_b^+(n):\; U_r^{b}(T(i)-) >b\}, \quad n \geq 1.
	\end{eqnarray*}
	
%\red{[moved here]}
It is clear that the strategy $\pi^b := (L_r^b(t); t \geq 0 )$, for $b \geq 0$, is admissible for the first problem defined in Section \ref{dividends-strategy}. Its expected NPV of dividends  is given by 
	\begin{align}\label{vf}
		v_b(x):=\E_x\Big(\int_{[0,\tau_0^b]}{\rm e}^{-qt}\diff L_r^b(t) + \rho {\rm e}^{-q \tau_0^b}\Big), \quad x \geq 0,
	\end{align}
%\red{JL: Forgot to put this, and so I added.}
where 
	\begin{align*}
	\tau_0^b := \inf \{ t > 0: U_r^b(t) < 0 \}.
	\end{align*}
	
%\red{[moved here]}
For the second problem, we want to show the optimality of an extension of the above strategy with additional classical reflection (capital injection) below at $0$, say $\bar{\pi}^{b}$, with a suitable Parisian reflection level $b \geq 0$. Namely, dividends are paid whenever the surplus process is above $b$ at dividend-decision times, while it is pushed upward by capital injection whenever it attempts to down-cross zero.  The controlled process, which we define formally below, is again the dual of a process considered in \cite{APY}.
	
	%\red{[added this]}
We construct the process $U_r^{0,b}$ with additional (classical) reflection below as follows.  Let
	\begin{align*}
		Y(t) := X(t) + R(t) \quad \textrm{where } R (t):= (-\inf_{0 \leq s \leq t} X(s)) \vee 0, \quad t \geq 0,
	\end{align*}
	be the process reflected from below at $0$.  We have
	\begin{align}
		U_r^{0,b}(t) = Y(t), \quad 0 \leq t < \widehat{T}_b^{+} (1) \label{Y_matches}
	\end{align}
	where $\widehat{T}_b^{+}(1) := \inf\{T(i):\; Y(T(i)) > b\}$.
	The process then jumps downward by $Y(\widehat{T}_b^{+}(1))-b$ so that $U_r^{0,b}(\widehat{T}_b^{+}(1)) = b$. For $\widehat{T}_b^{+}(1) \leq t < \widehat{T}_b^{+}(2)  := \inf\{T(i) > \widehat{T}_b^{+}(1):\; U_r^{0,b}(T(i) -) > b\}$, $U_r^{0,b}(t)$ is the process reflected  at $0$ of  the process $X(t) - X(\widehat{T}_b^{+}(1)) +b$. %$X(t) + |Y^b(\widehat{T}_0^-(1))|$ \green{[this is probably not correct -- should be $X(t) - X(\widehat{T}_0^-(1))$?]}.
	The process $U_r^{0,b}$ can be constructed by repeating this procedure.
	It is clear that it admits a decomposition
	\begin{align*}
		U_r^{0,b}(t) = X(t) - L_r^{0,b}(t) + R_r^{0,b}(t), \quad t \geq 0,
	\end{align*}
	%\red{[How about something like $U_r^{0,b}(t)$ instead of $Y_t^{r,b}$?]}
	where $L_r^{0,b}(t)$ and $R_r^{0,b}(t)$ are, respectively, the cumulative amounts of Parisian and classical reflection until time $t$. %\red{JL: How about writing $L_r^{0,b}$ and $R_r^{0,b}$?}

	It is clear that the strategy $\bar{\pi}^b := \{(L_r^{0,b}(t), R_r^{0,b}(t)); t \geq 0 \}$, for $b \geq 0$, is admissible for the second problem described in Section \ref{second_prob}. Its expected NPV is given by 
	\begin{align} \label{v_pi}
		u_b(x) := \mathbb{E}_x \left( \int_{[0, \infty)} {\rm e}^{-q t} \diff L_r^{0,b}(t) - \beta \int_{[0, \infty)} {\rm e}^{-q t} \diff R^{0,b}_r(t)\right), \quad x \geq 0.
	\end{align}
%	where $L_r^b$ and \red{$R_b^r$} are the processes studied given in Section \ref{PR} and studied in \cite{APY}. Here, we add the superscripts to stress the dependence on the reflection level $b$, which we aim to choose.

%\red{[JL: I combined these two sections.  Is this ok?]}
\subsection{Computation of the expected NPVs \eqref{vf} and \eqref{v_pi}}
%\section{Review of Scale Functions and some Fluctuation Identities.}
	%\subsection{Review of scale functions}  \label{section_scale_functions}
The expected NPVs of dividends (minus capital injection) as in \eqref{vf} and \eqref{v_pi} can be computed directly by using the fluctuation theory.  Toward this end, we first review the scale function.

%In this section we review here the scale function and its applications on the spectrally positive \lev process. As we need to deal with the fluctuation of the processes  $X$, we define the scale function here.
	Fix $q > 0$. We use $W^{(q)}$ for the scale function of the spectrally negative \lev process $-X$.  This is the mapping from $\R$ to $[0, \infty)$ that takes value zero on the negative half-line, while on the positive half-line it is a continuous and strictly increasing function that is defined by its Laplace transform:
	\begin{align} \label{scale_function_laplace}
		\begin{split}
			\int_0^\infty  \mathrm{e}^{-\theta x} W^{(q)}(x) \diff x &= \frac 1 {\psi(\theta)-q}, \quad \theta > \Phi(q),
		\end{split}
	\end{align}
	where $\psi$ is as defined in \eqref{lk} and
	\begin{align}
		\begin{split}
			\Phi(q) := \sup \{ \lambda \geq 0: \psi(\lambda) = q\} . 
		\end{split}
		\label{def_varphi}
	\end{align}
%	In particular, when $q=0$, we shall drop the superscript. \red{[if we assume $q > 0$, we can delete this?]}
	%By the strict  convexity of $\psi_Y$, we derive the inequality $\varphi(q) > \Phi(q) > 0$ for $q > 0$ and  $\varphi(q) \geq \Phi(q) \geq 0$ for $q = 0$.
	We also define, for $x \in \R$, %\red{added $\overline{Z}$}%\red{JL: removed $\overline{Z}$ because it is not used anywhere for the SP case.}
	\begin{align*}
		\overline{W}^{(q)}(x) &:=  \int_0^x W^{(q)}(y) \diff y, \\ 
		Z^{(q)}(x) &:= 1 + q \overline{W}^{(q)}(x), \\
		\overline{Z}^{(q)}(x) &:= \int_0^x Z^{(q)} (z) \diff z = x + q \int_0^x \int_0^z W^{(q)} (w) \diff w \diff z.
	\end{align*}
	Because $W^{(q)}(x) = 0$ for $-\infty < x < 0$, we have
	\begin{align}
		\overline{W}^{(q)}(x) = 0,\quad Z^{(q)}(x) = 1,
		  \quad \textrm{and} \quad \overline{Z}^{(q)}(x) = x, \quad x \leq 0.  \label{z_below_zero}
	\end{align}
	%\red{[JL: remoe the second one?  It is I guess not used anywhere.]}
	%\red{[moved the second line here]}
	%that can be obtained by performing an integration of identity (\ref{RLqp}).
	%	which can be proven by showing that the Laplace transforms on both sides are equal.
	
%\red{[added this.  Used later.]
If we define $\tau_0^- := \inf \left\{ t \geq 0: X(t) < 0 \right\}$ and $\tau_b^+ := \inf \left\{ t \geq 0: X(t)>  b \right\}$ for any $b > 0$, 
then, for $x \geq 0$,
\begin{align}
\begin{split}
\E_x \left( {\rm e}^{-q \tau_0^-} 1_{\left\{ \tau_b^+ > \tau_0^- \right\}}\right) &= \frac {W^{(q)}(b-x)}  {W^{(q)}(b)}, \\
 \E_x \left( {\rm e}^{-q \tau_b^+} 1_{\left\{ \tau_b^+ < \tau_0^- \right\}}\right) &= Z^{(q)}(b-x) -  Z^{(q)}(b) \frac {W^{(q)}(b-x)}  {W^{(q)}(b)}.
\end{split}
 \label{laplace_in_terms_of_z}
\end{align}
	
	\begin{remark} \label{remark_smoothness_zero}
		%\begin{enumerate}
		%	\item If $X$ is of unbounded variation or the \lev measure is atomless, it is known that $W^{(q)}$ is $C^1(\R \backslash \{0\})$; see, e.g.,\ \cite[Theorem 3]{Chan2011}. \red{[remove this item? It is not needed for this paper. Only the continuity of $W$ will do.]}
			%\item 
			Regarding the asymptotic behaviors near zero, as in Lemmas 3.1 and 3.2 of \cite{KKR},
			\begin{align}\label{eq:Wqp0}
				\begin{split}
					W^{(q)} (0) &= \left\{ \begin{array}{ll} 0 & \textrm{if $X$ is of unbounded
							variation,} \\ \frac 1 {c} & \textrm{if $X$ is of bounded variation,}
					\end{array} \right. \\
					W^{(q)\prime} (0+) &:= \lim_{x \downarrow 0}W^{(q)\prime} (x) =
					\left\{ \begin{array}{ll}  \frac 2 {\sigma^2} & \textrm{if }\sigma > 0, \\
						\infty & \textrm{if }\sigma = 0 \; \textrm{and} \; \Pi(0,\infty)= \infty, \\
						\frac {q + \Pi(0,\infty)} {c^2} &  \textrm{if }\sigma = 0 \; \textrm{and} \; \Pi(0,\infty) < \infty.
					\end{array} \right.
				\end{split}
			\end{align}
			On the other hand, as in Lemma 3.3 of \cite{KKR},
			\begin{align}
				\begin{split}
					{\rm e}^{-\Phi(q) x}W^{(q)} (x) \nearrow \psi'(\Phi(q))^{-1}, \quad \textrm{as } x \uparrow \infty.
				\end{split}
				\label{W_q_limit}
			\end{align}
			%where in the case $\psi'(0+) = 0$, the right hand side, when $q=0$,  is understood to be infinity. \red{[if we assume $q > 0$, we can delete this?]}
			%\item As in (8.22) and Lemma 8.2 of \cite{K}, $ {W^{(q)\prime}(y+)} / {W^{(q)}(y)} \leq  {W^{(q)\prime}(x+)} / {W^{(q)}(x)}$ for  $y > x > 0$.
		%	In all cases, $W^{(q)\prime}(x-) \geq W^{(q)\prime}(x+)$ for all $x >0$. \red{[not used? remove?]}
		%\end{enumerate}
	\end{remark}
	Along this work we also define, for $q, r > 0$ and $x \in \R$,
	\begin{align}
		J^{(q,r)}(x) &:={\rm e}^{\Phi(q+r) x} \left( 1 -r \int_0^{x} {\rm e}^{-\Phi(q+r) z} W^{(q)}(z) \diff z	\right) > 0, \label{z1}\end{align}
		where the positivity holds because, by \eqref{scale_function_laplace},
			%{\color{blue}[Kazu: Following your suggestion I changed $Z^{(q,r)}(x)$ to $J^{(q,r)}(x)$, is this ok?]}
\begin{align*}
r\int_0^x {\rm e}^{-\Phi(q+r)z}W^{(q)}(z)\diff z<r\int_0^{\infty}{\rm e}^{-\Phi(q+r)z}W^{(q)}(z) \diff z=\frac{r}{(q+r)-q}=1,
\end{align*}
	and define
	%\red{JL: Should we just define ${\color{blue}{\color{blue}J^{(q,r)}(x)}}$ as the general case with $\theta$ is not used anywhere? In that case, we may just use something like $Z^{(q)} (x, r)$? It is just that ${\color{blue}{\color{blue}J^{(q,r)}(x)}}$ is too long and  hard to see?}
	\begin{align}
		Z^{(q,r)}(x) &:=\frac{r}{r+q} Z^{(q)}(x)
		+\frac{q}{r+q} J^{(q,r)}(x) \label{z2}.
	\end{align}
%\red{[added this] 
Note that
\begin{align}
 Z^{(q,r)\prime}(x) = \frac {q} {r+q} \Phi(q+r) J^{(q,r)}(x), \quad x \in \R. \label{Z_q_r_der}
\end{align}%}
%\red{[JL: I just realized that the derivative of $Z^{(q,r)}$ appears and so it makes more sense to use $Z^{(q,r) \prime}$.  Sorry, is it ok to change back instead of $\frac {q} {r+q} \Phi(q+r) Z^{(q)}(x,r)$ below and use \eqref{Z_q_r_der} when necessary?]}

%\red{JL: I made this a lemma.}

The expected NPVs \eqref{vf} and \eqref{v_pi} can be written concisely by the scale functions defined above.  All the fluctuation identities required here are essentially computed in \cite{APY}; they studied the spectrally negative \lev case where it is reflected from below at Poisson arrival times, and also its variation with additional classical reflection from above. Our processes $U_r^b$ and $U^{0,b}_r$ are the dual of these processes and hence their results can be directly used.
%Now we will use the results in \cite{APY} to give some of the fluctuation identities that will be used throughout this work. 
\begin{lemma} For all $b \geq 0$ and $x \geq 0$,
\begin{align}\label{vf3}
v_b(x)&=\frac{H^{(q,r)}(b) +\rho}{Z^{(q,r)}(b)}Z^{(q,r)}(b-x)-H^{(q,r)}(b-x), \\
%\end{align}
%and
%\begin{align}
\label{vf_db_a}
u_b(x)&=\left(\frac{r Z^{(q)} (b)}{ r+q}-\beta\right)\frac{Z^{(q,r)}(b-x)}{Z^{(q,r)\prime}(b)}-H^{(q,r)}(b-x),
\end{align}
%\red{JL: Should we just write $h'$ as $r Z^{(q)} (x) / (r+q)$?}
	where %\red{[the reader may wonder why there should be tilde. How about using something like $H^{(q,r)}$?]}
	\begin{align*}
		H^{(q,r)}(y)&:=\frac{r}{r+q}\left(\overline{Z}^{(q)}(y)+\frac{\psi'(0+)}{q}\right), \quad y \in \R. %\\
		%Z^{(q,r)}(x)&=\frac{r}{(r+q)}Z^{(q)}(x)+\frac{q}{(q+r)}{\color{blue}{\color{blue}J^{(q,r)}(x)}}.\\
		%{\color{blue}{\color{blue}J^{(q,r)}(x)}}&=e^{\Phi(q+r)x}\left(1-r\int_0^xe^{-\Phi(q+r)z}W^{(q)}(z)dz\right).
	\end{align*}
	%\red{[JL: It is better to write $\frac q {r+q} \Phi(q+r) Z^{(q)}(x,r)$ instead of $Z^{(q,r)'}$?]}
\end{lemma}
\begin{proof} %\red{[let's simplify this.  It is really straightforward and so I guess we don't even need to introduce $\tilde{L}^{(r)}$?]}
	%It is not difficult to see that
%	\begin{align*}
	%	\E_x\left[\int_{[0,\tau_0^{\pi}]}e^{-qt}\diff L_r^b(t)\right]=\E_{b-x}\left[\int_{[0,\tilde{\tau}_b^+]}e^{-qt} \diff \tilde{L}^{(r)}(t)\right],
%	\end{align*}
	%where $\tilde{L}^{(r)}$ represents the accumulated Poissonian bail-outs of the spectrally negative L\'evy process $-X$ in the context of \cite{APY} and $\tilde{\tau}_b^+=\inf\{t>0:-U^{r,b}(t)>b\}$.
	\par Using Corollaries 3.1 (ii) and 3.2 (ii) in \cite{APY}, we obtain that, for $x \geq 0$, % \red{[I guess this expression also holds for $x > b$ (I guess we can modify the other paper?)]} {\color{blue} Kazu: Yes it does, it follows from the below computation.}
	\begin{align}\label{fi1}
	\begin{split}
			\E_x\Big( \int_{[0,\tau_0^b]} {\rm e}^{-qt}\diff L_r^b(t)\Big) &=\frac{H^{(q,r)}(b)}{Z^{(q,r)}(b)}Z^{(q,r)}(b-x)-H^{(q,r)}(b-x), \\
			\E_x\Big( {\rm e}^{-q\tau_0^b} \Big) & =\frac{Z^{(q,r)}(b-x)}{Z^{(q,r)}(b)},
			\end{split}
	\end{align}
which show \eqref{vf3}.

%\red{JL: I guess we should use the version where $-$ is removed for $\tilde{h}$?}{\color{blue} Do you mean
%\[
%\E_x\left[\int_{[0,\tau_0^{\pi}]}e^{-qt}\diff L_r^b(t)\right]=\frac{{\color{blue}H^{(q,r)}}(b)}{Z^{(q,r)}(b)}Z^{(q,r)}(b-x)-{\color{blue}H^{(q,r)}}(b-x),
%\]
%with
%\[
%		{\color{blue}H^{(q,r)}}(x)=\frac{r}{(r+q)}\left(\bar{Z}^{(q)}(x)+\frac{\psi'(0+)}{q}\right)? %\\
%\]} \red{yes.}
	On the other hand, we have using Corollary 3.4 in \cite{APY} that
	\begin{align}\label{fi2_a}
		\E_x\Big( \int_{[0, \infty)} {\rm e}^{-qt}\diff L_r^{0,b}(t) \Big) =\frac{r }{ r+q} \frac{Z^{(q,r)}(b-x)}{Z^{(q,r)\prime}(b)} Z^{(q)} (b)-H^{(q,r)}(b-x), \quad x\geq 0,
	\end{align}
	and using Corollary 3.3 in \cite{APY} we obtain
	\begin{align}\label{fi2_b}
		\E_x\Big( \int_{[0, \infty)} {\rm e}^{-qt}\diff R_r^{0,b}(t) \Big)= \frac{Z^{(q,r)}(b-x)}{Z^{(q,r)\prime}(b)}, \quad x\geq 0.
	\end{align}
	Subtracting the latter (times $\beta$) from the former, we have \eqref{vf_db_a}.
	%{\color{blue}Kazu: Can you check if I oversimplified the proof please?}
	\end{proof}

It is noted that the expressions \eqref{vf3} and \eqref{vf_db_a} also hold for $x \geq b \geq 0$ with %\red{[simplified the expression a bit below]}
\begin{align*}
			v_b(x) &=(r+q)^{-1} \Big[ \frac{H^{(q,r)}(b) +\rho}{Z^{(q,r)}(b)}\Big( r 
		+q {\rm e}^{(b-x)\Phi(q+r)} \Big) - r \left(b-x+\frac{\psi'(0+)}{q}\right) \Big], \\
		u_b(x) &=(r+q)^{-1} \Big[\frac{1}{Z^{(q,r) \prime}(b)}\left(\frac{r Z^{(q)} (b)}{ r+q}-\beta\right)\Big( r
		+q  {\rm e}^{(b-x)\Phi(q+r)} \Big) -r\left(b-x+\frac{\psi'(0+)}{q}\right) \Big].\end{align*}

\section{Solutions to the Optimal dividend problem with Poissonian dividend-decision times} \label{section_first_prob}
In this section, we solve the first problem defined in Section \ref{dividends-strategy}. Focusing on the periodic barrier strategies $(\pi^b; b \geq 0)$, we shall first identify the candidate barrier $b^*$ so that the expected NPV $v_{b^*}$, if $b^* > 0$, gets smoother at $b^*$.  We shall then show its optimality by verifying that $v_{b^*}$ solves the required variational inequalities.

%\subsection{Candidate value function.}\label{cvf_div} \red{[remove this subsection?]}
%In this section we will compute the value function for the dividend problem at Poissonian times. So using \eqref{fi1} it is not difficult to see that the value function takes the following form:
%\begin{align}\label{vf3}
%v_b(x)=\frac{{\color{blue}H^{(q,r)}}(b)}{Z^{(q,r)}(b)}Z^{(q,r)}(b-x)-{\color{blue}H^{(q,r)}}(b-x)\qquad\text{for $x\geq0$},
%\end{align}
%and for $x>b$
%\begin{align}\label{vf4}
%v_b(x)=-r\frac{(b-x)(q+r)+(1-e^{\Phi(q+r)(b-x)})\psi'(0+)}{(q+r)^2}+\left(\frac{r}{(r+q)}+\frac{q}{(r+q)}e^{\Phi(q+r)(b-x)}\right)v_b(b).
%\end{align}

%\red{[JL: By simply setting $x \geq b$, we get this. Maybe this way is cleaner]
%Here in particular, for $x \geq b$,
%\begin{align*}
%			v_b(x) &=-\frac{r}{r+q}\left((b-x)+\frac{\psi'(0+)}{q}\right) +\frac{{\color{blue}H^{(q,r)}}(b)}{Z^{(q,r)}(b)}\Big( \frac{r}{r+q} 
%		+\frac{q}{r+q} e^{(b-x)\Phi(q+r)} \Big).
%	\end{align*}
%	Then,
	
\subsection{Smooth fit}\label{sf_d} %\red{[I changed it so that we refrain from using $b^*$ until we define as the candidate threshold. OK?]}
Motivated by many papers in the literature (see, e.g., \cite{AGS2007, AGS2008,ASW2011,ATW,BKY,BKY2}), we shall choose the barrier so that \emph{the degree of smoothness there increases by one}.  Differently from the classical dual model as in \cite{BKY} where the value function becomes $C^1 (0, \infty)$ (resp.\ $C^2 (0, \infty)$) for the case $X$ is of bounded (resp.\ unbounded) variation, we shall see in this case that we will have $C^2 (0, \infty)$ (resp.\ $C^3 (0, \infty)$) for the case $X$ is of bounded (resp.\ unbounded) variation.

%\red{[moved here and named the condition. OK?] 
Here, we shall show that the desired smoothness at $b$ is satisfied on condition that 
\begin{equation}\label{opt_thres}
\mathfrak{C}_b: -\frac{H^{(q,r)}(b)+\rho}{Z^{(q,r)}(b)}=\frac{1}{\Phi(q+r)}.
\end{equation}
%holds for $b^* = b$.

%In this section we will obtain the candidate $b^*$ for the optimal threshold. In order to do it we will look for smoothness of the value function. To this end we note that for $x<b$
For all $b > 0$ and $x \in (0, \infty) \backslash \{b\}$, by differentiating \eqref{vf3}, %[I wrote this way, because the expressions are valid for $x > b$]
\begin{align}\label{vf_der1}
\begin{split}
v_b'(x)&=\frac{r}{r+q}Z^{(q)}(b-x)-\frac{H^{(q,r)}(b)+\rho}{Z^{(q,r)}(b)}\frac{q}{r+q}\Phi(q+r)J^{(q,r)}(b-x),\\
v_b''(x)&=-\frac{rq}{r+q}W^{(q)}(b-x)+\frac{H^{(q,r)}(b)+\rho}{Z^{(q,r)}(b)}\frac{q}{r+q}\left[\Phi^2(q+r) J^{(q,r)}(b-x)-r \Phi(q+r) W^{(q)}(b-x)\right],\\%\qquad\text{and}\\
v_b'''(x-)&=\frac{rq}{r+q}W^{(q)\prime}((b-x)+) -\frac{H^{(q,r)}(b)+\rho}{Z^{(q,r)}(b)}\frac{q}{r+q} \\ &\times \left[\Phi^3(q+r) J^{(q,r)}(b-x)-r \Phi^2(q+r) W^{(q)}(b-x)- r \Phi(q+r) W^{(q)\prime}((b-x)+)\right],
\end{split}
\end{align}
%\red{[JL: Above, instead of multiplying $\Phi(q+r)$ repeatedly, I used $\Phi(q+r)^2$ and $\Phi(q+r)^3$ ok?]}
where in particular, for $x>b$,
\begin{align}\label{vf_der2}
\begin{split}
	v_b'(x) &=\frac{r}{r+q} - \Phi(q+r) \frac{H^{(q,r)}(b)+\rho}{Z^{(q,r)}(b)}  \frac{q}{r+q} {\rm e}^{\Phi(q+r)(b-x)}, \\
	v_b''(x) &= \Phi^2(q+r) \frac{H^{(q,r)}(b)+\rho}{Z^{(q,r)}(b)}  \frac{q}{r+q} {\rm e}^{\Phi(q+r)(b-x)}, \\
	v_b'''(x) &=-\Phi^3(q+r) \frac{H^{(q,r)}(b)+\rho}{Z^{(q,r)}(b)}  \frac{q}{r+q} {\rm e}^{\Phi(q+r)(b-x)}.
	\end{split}
\end{align}
%\begin{align}\label{vf_der2}
%v_b'(x)&=\frac{-r}{(r+q)^2}(-(q+r)+e^{\Phi(q+r)(b-x)}\psi'(0+)\Phi(q+r))-\frac{q\Phi(q+r)}{(r+q)}e^{\Phi(q+r)(b-x)}v_b(b),\notag\\
%v_b''(x)&=\left(\frac{r\Phi^2(q+r)\psi'(0+)}{(r+q)^2}+\frac{q\Phi^2(q+r)}{(r+q)}v_b(b)\right)e^{\Phi(q+r)(b-x)}\qquad\text{and,}\notag\\
%v_b'''(x)&=-\Phi(q+r)\left(\frac{r\Phi^2(q+r)\psi'(0+)}{(r+q)^2}+\frac{q\Phi^2(q+r)}{(r+q)}v_b(b)\right)e^{\Phi(q+r)(b-x)}.
%\end{align}
%\red{[I guess we don't need this?]}
%Now we note from \eqref{vf3} that \red{JL: What is the role of this?  Used anywhere?}
%\begin{equation}\label{vf_b}
%v_b(b)=-\frac{r}{(r+q)}\frac{\psi'(0+)}{q}+\frac{{\color{blue}H^{(q,r)}}(b)}{Z^{(q,r)}(b)}.
%\end{equation}

(i) First, it is immediate that $v_b$ is continuously differentiable for any choice of $b > 0$ with
\begin{align*}
v_b'(b+) = v_b'(b-)=\frac{r}{r+q}-\frac{q}{r+q}\Phi(q+r)\frac{H^{(q,r)}(b)+\rho}{Z^{(q,r)}(b)}.
\end{align*}
%} \red{[and move to the second derivative?]}

%This implies that
%\begin{align*}
%v_b'(b+)&=\frac{r}{(r+q)}-\frac{q}{(r+q)}\Phi(q+r)\frac{{\color{blue}H^{(q,r)}}(b)}{Z^{(q,r)}(b)},\qquad\text{and}\\
%v_b'(b-)&=\frac{r}{(r+q)}-\frac{r\psi'(0+)\Phi(q+r)}{(r+q)^2}-\frac{q}{(r+q)}\Phi(q+r)v_b(b)\\
%&=\frac{r}{(r+q)}-\frac{q}{(r+q)}\Phi(q+r)\frac{{\color{blue}H^{(q,r)}}(b)}{Z^{(q,r)}(b)}.
%\end{align*}
%Therefore the value function $v_b$ is continuously differentiable. 

(ii) For the second derivative, %\red{[change from $W(0+)$ to $W(0)$. OK?]} 
\begin{align*}
v_b''(b+)%&=\left(\frac{r\Phi^2(q+r)\psi'(0+)}{(r+q)^2}+\frac{q\Phi^2(q+r)}{(r+q)}v_b(b)\right)\\
&=\frac{q\Phi^2(q+r)}{r+q}\frac{H^{(q,r)}(b)+\rho}{Z^{(q,r)}(b)} \qquad\text{and}\\
v_b''(b-)&=-\frac{rq}{r+q}W^{(q)}(0)+\frac{H^{(q,r)}(b)+\rho}{Z^{(q,r)}(b)}\frac{q}{r+q}\Phi(q+r)\big[ \Phi(q+r)-rW^{(q)}(0) \big].
\end{align*}
Hence, in order for the function $v_b$ to be twice continuously differentiable, we need to ask that %\red{[simplified the left-hand side a bit.]}
\[
\frac{rq}{q+r}W^{(q)}(0) \Big[1+\frac{H^{(q,r)}(b)+\rho}{Z^{(q,r)}(b)}\Phi(q+r) \Big]=0.
\]
This means, in view of \eqref{eq:Wqp0}, that while the twice continuous differentiability automatically holds for the unbounded variation case, for the bounded variation case it holds if and only if $\mathfrak{C}_b$ holds.

(iii) For the unbounded variation case, we will look for the continuity of the third derivative of the function %[avoid to say ``value function" here.] 
$v_b$. Using \eqref{eq:Wqp0}, \eqref{vf_der1}, and \eqref{vf_der2},  we obtain
\begin{align*}
	v_b'''(b+)%&=-\Phi(q+r)\left(\frac{r\Phi^2(q+r)\psi'(0+)}{(r+q)^2}+\frac{q\Phi^2(q+r)}{(r+q)}v_b(b)\right)
	&=-q\frac{\Phi^3(q+r)}{r+q}\frac{H^{(q,r)}(b)+\rho}{Z^{(q,r)}(b)}, \\
v_b'''(b-)&=\frac{rq}{r+q}W^{(q)\prime}(0+)
-\frac{H^{(q,r)}(b)+\rho}{Z^{(q,r)}(b)}\frac{q}{r+q}\Phi(q+r)\left[\Phi^2(q+r)-rW^{(q)\prime}(0+)\right]. 
%\end{align*}
%and using and \eqref{vf_b} %\red{[remove the middle equality below?]}
%\begin{align*}
\end{align*}
Therefore the value function $v_b$ will have a continuous third derivative if \[
\frac{rq}{r+q}W^{(q)\prime}(0+) \Big[ 1+\frac{H^{(q,r)}(b)+\rho}{Z^{(q,r)}(b)} \Phi(q+r)\Big] =0,
\]
which holds if and only if $\mathfrak{C}_b$ holds.

%\red{[I added a lemma below. ] 
We shall now summarize the results obtained above.
\begin{lemma} \label{smooth_fit_prob1}Suppose $b > 0$ is such that the condition $\mathfrak{C}_b$ as in \eqref{opt_thres} is satisfied. Then, $v_{b}$ is $C^2 (0, \infty)$ for the case $X$ is of bounded variation, while it is $C^3 (0, \infty)$ for the case $X$ is of unbounded variation.
\end{lemma}

%Hence in the general case we need to ask that the threshold $b^*$ satisfies that TODO

%\red{[JL: Sorry, I did not have time to check the sign for $\tilde{h}$ but something is wrong here.  The right hand side is positive while the left hand side is negative.]}
%{\color{blue}[ Kazu: I believe that $h^{(q,r)}$ can be negative depending on the sign of $\psi'(0+)$.
%That is why Lemma 3.1 has a condition for the existence of a $b^*>0$ such that \eqref{opt_thres} holds. And it is related to te sign of $\psi'(0+)$.]}

%So using \eqref{opt_thres} we can write the value function in the following form

% {\color{blue} [Kazu: Removed the case $x>b^*$, and left only \eqref{vf5} is this ok?]	.}
%And for $x>b$ \red{[\eqref{vf6} may not be useful?]}
%\begin{align}\label{vf6}
%	v_{b^*}(x)=-r\frac{(b^*-x)(q+r)+(1-e^{\Phi(q+r)(b^*-x)})\psi'(0+)}{(q+r)^2}+\left(\frac{r}{(r+q)}+\frac{q}{(r+q)}e^{\Phi(q+r)(b^*-x)}\right)v_{b^*}(b^*),
%\end{align}
%with
%\begin{equation}\label{opt_vf}
%v_b(b)=-\frac{r}{(r+q)}\frac{\psi'(0+)}{q}-\frac{1}{\Phi(q+r)}.
%\end{equation}
\subsection{Selection of the candidate barrier $b^*$.}
Below we obtain a necessary and sufficient condition for the existence of $b$ that satisfies $\mathfrak{C}_b$ as in \eqref{opt_thres}. %solution to \eqref{opt_thres}. %So we have the following result:

%\red{JL: Let us set $b^* = 0$ when \begin{equation}
%\psi'(0+)\geq-\frac{q}{r\Phi(q+r)}(r+q); \label{b_zero_criteria}
%\end{equation}
%namely, it takes all the money and run at the first exponential time. the optimality should still hold.
%}
\begin{lemma}\label{lemma_opt_thres}
There exists a unique solution $\tilde{b} >0$ to the equation \eqref{opt_thres} if and only if
\begin{equation}\label{cond_opt_thres}
\psi'(0+) < - \frac q r (q+r) \Big(\rho + \frac 1 {\Phi(q+r)} \Big) =: I_{r,q}.
\end{equation}
%And $b^*=0$ if the equality holds in \eqref{cond_opt_thres}. \red{I think this last sentence is confusing.  Basically, if \eqref{cond_opt_thres} holds then there exist a unique $b^*$ such that \eqref{opt_thres}.  We can stop like that and later define $b^* = 0$ if it does not hold?}
\end{lemma}
%\red{JL: the right hand side appears often and so I defined a short-hand notation (it does not have to be $I_{r,q}$).}
\begin{proof}
%Let us consider the following function
First it is clear that the condition $\mathfrak{C}_b$ is equivalent to the condition $f(b) = 0$ where %\red{[simplified the expression a bit]}
\begin{align}
%f(b):=\overline{Z}^{(q)}(b)+\frac{\psi'(0+)}{q}+\frac{1}{\Phi(q+r)} \Big[ Z^{(q)}(b)+\frac{q}{r}{\color{blue}J^{(q,r)}(b)} \Big]. 
f(b):=\overline{Z}^{(q)}(b)+\frac{\psi'(0+)}{q} +\frac {q+r} r \rho+ \frac{1}{\Phi(q+r)} \frac {r+q} r Z^{(q,r)}(b). \label{def_f}
\end{align}
Differentiating this and by \eqref{Z_q_r_der},
\begin{align*}
	f'(b)%&=Z^{(q)}(b)+\frac{q}{\Phi(q+r)}W^{(q)}(b)+\frac{q}{r\Phi(q+r)}\big[\Phi(q+r){\color{blue}J^{(q,r)}(b)}-rW^{(q)}(b) \big]\\
	&=Z^{(q)}(b)+ \frac{q}{r}J^{(q,r)}(b)>0.
\end{align*}
Hence, the function $f$ is strictly increasing, and we note that, by \eqref{W_q_limit}, $\lim_{b\to\infty}f(b)=\infty$.

Therefore there exists a unique point $\tilde{b}>0$ such that $f(\tilde{b})=0$ if and only if $f(0)<0$, which is equivalent to \eqref{cond_opt_thres} because $f(0)=\frac{\psi'(0+)}{q}  +\frac {q+r} r (\rho +\frac{1}{\Phi(q+r)})$.
%is equivalent to the following condition
%	\[
%	\psi'(0+)\leq-\frac{q}{r\Phi(q+r)}(r+q).
%	\]
%	And we have that $b^*=0$ if and only if
%	\[
%	\psi'(0+)=-\frac{q}{r\Phi(q+r)}(r+q).
%	\]
%The proof follows by noting that if $b^*$ is the unique root of $f$ that satisfies \eqref{opt_thres}.
\end{proof}

In view of Lemma \ref{lemma_opt_thres}, we will take, as the candidate optimal barrier $b^*$, the unique root of \eqref{opt_thres} if \eqref{cond_opt_thres} holds. For the case in which
\begin{equation}
\psi'(0+)\geq I_{r,q} %-\frac{q}{r\Phi(q+r)}(r+q), 
\label{b_zero_criteria}
\end{equation}
holds, we will take the candidate optimal barrier as $b^*=0$; namely, the corresponding strategy \emph{takes all the money and runs at the first opportunity}, which occurs at the first Poissonian dividend-decision time. As has been observed in \cite{ATW} (when $\rho = 0$), this can happen even when $\E X_1 = -\psi'(0+) > 0$, while in the classical model this happens if and only if $\E X_1 \leq 0$ (see \cite{BKY}).

\begin{remark} Suppose $\rho = 0$.
In view of \eqref{cond_opt_thres}, the threshold $I_{r,q}$ vanishes in the limit as $r \rightarrow \infty$. In other words, the criterion for $b^* = 0$ converges to that in the classical case as the frequency of dividend-decision opportunities increases to infinity.

On the other hand, as $r \rightarrow 0$, $I_{r,q} \rightarrow -\infty$, which means  $b^* = 0$ for small enough $r > 0$.  This suggests to take all the money and run at the first opportunity if one needs to expect a long time until the next dividend-decision time. 
%This is consistent with the existing result that in the classical case that it is optimal to spend all the money and run if and only if $\E X_1 = -\psi'(0+) \leq 0$; see \cite{BKY}.
\end{remark}
\subsection{Verification.}\label{ver_div}

With $b^* \geq 0$ defined above, we shall now show the optimality of the obtained periodic barrier strategy %reflection strategy at Poissonian dividend decision times 
 $\pi^{b^*}$.

%\red{[moved this here]
For the case $b^* > 0$, because $b^*$ satisfies $\mathfrak{C}_b$, the expected NPV \eqref{vf3} can be succinctly written
\begin{align}\label{vf5}
	v_{b^*}(x)=-H^{(q,r)}(b^*-x)-\frac{Z^{(q,r)}(b^*-x)}{\Phi(q+r)},\qquad\text{for $x\geq 0$}.
\end{align}
%\red{moved this here as well. OK?}
On the other hand, when $b^* = 0$, we have \begin{align} \label{value_function_zero}
\begin{split}
v_{b^*}(x) =v_0(x)&=-H^{(q,r)}(-x)+\frac{H^{(q,r)}(0) +\rho}{Z^{(q,r)}(0)}Z^{(q,r)}(-x) \\
%&= \frac r {r+q} (x - \frac {\psi'(0+)} q ) + \frac r {r+q} \frac {\psi'(0+)} q  \Big( \frac r {r+q} + \frac q {r+q} e^{- \Phi(q+r)x}\Big) \\
%&= \frac r {r+q} \Big[ (x - \frac {\psi'(0+)} q ) +  \frac {\psi'(0+)} q  \Big( \frac r {r+q} + \frac q {r+q} e^{- \Phi(q+r)x}\Big) \Big] \\
&= \frac r {r+q} \Big[ x -  \frac {\psi'(0+)} {r+q}  \Big( 1 -  {\rm e}^{- \Phi(q+r)x}\Big) \Big] + \rho \Big( \frac r {r+q} + \frac q {r+q} {\rm e}^{-\Phi(q+r) x}\Big) \\
&= \frac r {r+q} \Big[ x -  \Big(\frac {\psi'(0+)} {r+q} - \rho \Big)   +  \Big(\frac {\psi'(0+)} {r+q} + \frac {\rho q} r\Big) {\rm e}^{- \Phi(q+r)x}\Big) \Big]  \qquad\text{for $x\geq 0$}.
\end{split}
\end{align}
%\red{[remove the two middle lines above?]}

%We shall now show
Our main result of this section is given as follows.
\begin{theorem}\label{ver_theo}
The periodic barrier strategy $\pi^{b^*}$ is optimal, and the value function is given by $v(x)=v_{b^*}(x)$ for all $0\leq x<\infty$.
\end{theorem}
In order to prove Theorem \ref{ver_theo}, we shall provide the verification lemma and show that $v_{b^*}$ satisfies the stated conditions. We call a measurable function $g$ \emph{sufficiently smooth} if $g$ is $C^1 (0,\infty)$ (resp.\ $C^2 (0,\infty)$) when $X$ has paths of bounded (resp.\ unbounded) variation.
%Following Proposition 4 (ii) in \cite{APP2007}, we extend the domain of the function $v_\pi$, for all $\pi \in \mathcal{A}$, as in \eqref{vf}, to all $\mathbb{R}$ by setting $v_\pi(x):=v_\pi(0)+\beta x$ for $x < 0$. \red{JL: For this problem, I think this extension would not be necessary?  $\beta$ is not defined anyways?}
We let $\mathcal{L}$ be the operator acting on a sufficiently smooth function $g$, defined by
\begin{equation}
\begin{split}
\mathcal{L} g(x)&:= -\gamma g'(x)+\frac{\sigma^2}{2}g''(x) +\int_{(0,\infty)}[g(x + z)-g(x)-g'(x)z\mathbf{1}_{\{0<z<1\}}]\Pi(\mathrm{d}z). \label{generator}
\end{split}
\end{equation}

Here, we give a generalization of Lemma 3.4 in \cite{ATW}, for a general spectrally positive \lev process.
%Then we have the following result, {\color{blue} which is a generalization of Lemma 3.4 in \cite{ATW}.} %\red{I guess we need to acknowledge that B\&B did this and we are generalizing their results?}
\begin{lemma}[Verification lemma]
	\label{verificationlemma}
	Suppose $\hat{\pi} \in \mathcal{A}$ is such that $v_{\hat{\pi}}$ is sufficiently smooth on $(0,\infty)$, right-continuous at zero with
	\begin{align}
	v_{\hat{\pi}}(0+) = \rho, \label{v_0_continuity}
	\end{align}
	 %differentiable at zero \red{[not so sure differentiability is needed. Currently the function is not defined for $(-\infty, 0)$]}, 
	and satisfies
	\begin{align}
	\label{HJB-inequality}
	(\mathcal{L} - q)v_{\hat{\pi}}(x)+r\max_{0\leq l\leq x}\{l+v_{\hat{\pi}}(x-l)-v_{\hat{\pi}}(x)\}\leq 0,  \quad x > 0.
	\end{align} %\red{changed to $m$ because it would be confusing with the martingale $M$.}
	Then $v_{\hat{\pi}}(x)=v(x)$ for all $x\geq0$ and hence $\hat{\pi}$ is an optimal strategy.
\end{lemma}
\begin{proof}
	See Appendix \ref{Appendix_A}.
\end{proof}
\par In the rest of this section, we will show that our candidate value function $v_{b^*}$ satisfies the sufficient condition \eqref{HJB-inequality} (the condition \eqref{v_0_continuity} is clearly satisfied). Recall from Lemma \ref{smooth_fit_prob1} that $v_{b^*}$ is sufficiently smooth.  

\begin{lemma}\label{cond_b^*}
(i) Suppose $b^* > 0$. The function $v_{b^*}$ is strictly increasing and concave on $(0, \infty)$,
 and $v_{b^*}'(b^*)=1$.

(ii) Suppose $b^* = 0$. We have that $v_{b^*}'(x)\leq 1$ for all $x > 0$.
\end{lemma}
\begin{proof}
(i) Suppose $b^* > 0$. By \eqref{opt_thres} and \eqref{vf_der1}, for all $x > 0$, %\red{[simplified below a bit by cancelling $\Phi(q+r)$]}
\begin{align}
v_{b^*}'(x)=\frac{r}{r+q}Z^{(q)}(b^*-x)+\frac{q}{r+q}J^{(q,r)}(b^*-x) = Z^{(q,r)}(b^*-x)>0, \label{v_prime_b_star}
\end{align}
and for the second derivative, by \eqref{Z_q_r_der},
\begin{align*}
	v_{b^*}''(x)
	%&=-\frac{rq}{r+q}W^{(q)}(b^*-x)-\frac{q}{r+q} \left[\Phi(q+r){\color{blue}J^{(q,r)}(b^*-x)}-rW^{(q)}(b^*-x)\right]\\
	&=-\frac{q}{r+q}\Phi(q+r)J^{(q,r)}(b^*-x)<0.
	\end{align*}
%\red{JL: So I guess we don't need to do for $x > b^*$ separately.  And can remove the following?}Now for the case in which $x>b^*$, using \eqref{opt_thres}  and \eqref{vf_der2}, we obtain
%\begin{align}\label{vf_der3}
%v_{b^*}'(x)&=\frac{-r}{(r+q)^2} \big[ -(q+r)+e^{\Phi(q+r)(b^*-x)}\psi'(0+)\Phi(q+r) \big] -\frac{q\Phi(q+r)}{(r+q)}e^{\Phi(q+r)(b^*-x)}v_{b^*}(b^*)\notag\\
%&=\frac{r}{(r+q)}+\frac{q}{(r+q)}e^{\Phi(q+r)(b^*-x)}>0,
%\end{align}
%and for the second derivative, using the above identity,
%\begin{equation*}
%v_{b^*}''(x)=-\frac{q}{(r+q)}\Phi(q+r)e^{\Phi(q+r)(b^*-x)}<0.
%\end{equation*}
%In sum, we proved that for all $x>0$ $v'_{b^*}(x)>0$ and $v_{b^*}''(x)<0$, and therefore the function 
In other words, $v_{b^*}$ is strictly increasing and concave on $(0, \infty)$. In addition,
%For the last remaining part of the proof 
 by \eqref{v_prime_b_star}, we have $v_{b^*}'(b^*)=1$.

(ii) Suppose $b^*=0$. Differentiating \eqref{value_function_zero}, for $x > 0$,
\begin{align}
v_0'(x)&= \frac r {r+q} \Big[1 -  \Phi(q+r) \Big(\frac {\psi'(0+)} {r+q} + \frac {\rho q} r\Big) {\rm e}^{- \Phi(q+r)x} \Big], \label{der_v_0} \\
v_0''(x)&=  \frac r {r+q}   \Phi^2(q+r) \Big(\frac {\psi'(0+)} {r+q} + \frac {\rho q} r\Big) {\rm e}^{- \Phi(q+r)x}. \label{der_v_0_2}
\end{align}
In particular, 
\begin{align*}
v_0'(0+)&= \frac r {r+q}  \Big[ 1-   \Phi(q+r) \Big(\frac {\psi'(0+)} {r+q} + \frac {\rho q} r\Big) \Big],
\end{align*}
which is less than or equal to $1$ by \eqref{b_zero_criteria}.  In addition, we have $v_0'(x) \xrightarrow{x \uparrow \infty} r / (r+q)$ by \eqref{der_v_0}. 

(a) If $\psi'(0+) + \rho q (r+q)/r> 0$, then, by \eqref{der_v_0_2}, $v_0$ is strictly convex; hence $ v_0'(0+) < v_0' (x) < r/ (r+q) \leq 1$ for $x > 0$.  

(b) If $\psi'(0+) + \rho q (r+q)/r< 0$, then, by \eqref{der_v_0_2}, $v_0$ is strictly concave; hence $1 \geq v_0'(0+) >  v_0' (x) > r/ (r+q) >0$ for $x > 0$.

(c) If $\psi'(0+) + \rho q (r+q)/r= 0$, then $1 > v_0'(x) = r / (r+q) > 0$ for $x > 0$.
\end{proof}
  %The increasingness can be proven because we must have $v_0'(0+) > 0$ due to $v_0(0) = 0$ and $v_0(x) > 0$ uniformly in $x$.
%\par And so the proof follows.
%\red{This completes the proof.}
Next, by an application of Lemma \ref{cond_b^*} (i) and (ii), respectively, for $b^* > 0$ and $b^* = 0$, the following results are immediate.%we obtain the following result.
\begin{lemma}\label{cond_max_v}
	For $b^*\geq0$ we have that
	\begin{equation}\label{max_cond}
	\max_{0\leq l\leq x} \{ l+v_{b^*}(x-l)-v_{b^*}(x) \} =
	\begin{cases} 0 &\mbox{if } x \in[0,b^*], \\ 
	x-b^*+v_{b^*}(b^*)-v_{b^*}(x) & \mbox{if } x \in (b^*,\infty). \end{cases}
	\end{equation}
\end{lemma}
%
%\red{JL: This is I think immediate from Lemma \ref{cond_b^*} (i) and (ii), respectively for $b^* > 0$ and $b^* = 0$. I guess we can just say so?  Also, we should place this lemma right after Lemma \ref{cond_b^*}?} 
We shall  next show the following.
\begin{lemma}\label{generator_on_v}
If $b^* > 0$, we have %\red{[cleaned up the equation a bit below]}
\begin{equation}\label{gen_on_v}
(\mathcal{L}-q)v_{b^*}(x)=
\begin{cases} 0 &\mbox{if } x \in(0,b^*], \\ 
\displaystyle \frac{qr}{r+q} \Big( (b^*-x)+\frac{1-{\rm e}^{\Phi(q+r)(b^*-x)}}{\Phi(q+r)} \Big) & \mbox{if } x \in (b^*,\infty). \end{cases}
\end{equation}
%\red{[move this here]} 
If $b^* = 0$, we have %\red{[cleaned up the equation a bit below]}
\begin{align}
(\mathcal{L}-q) v_0(x)
= \frac r {r+q}  \Big[ - \Big(\frac {r \psi'(0+)} {r+q} + q \rho \Big)  (1- {\rm e}^{-\Phi(q+r)x})- qx \Big],
%&= \frac r {r+q}  \Big[ - qx  - r \frac {\psi'(0+)} {r+q} (1- e^{-\Phi(q+r)x})\Big], 
\quad x >0.\label{generator_zero_case}
\end{align}
\end{lemma}
\begin{proof} Suppose $b^* > 0$.
(i) By the proof of Theorem 2.1 in \cite{BKY}, we have that
\begin{align}
(\mathcal{L}-q)H^{(q,r)}(b^*-x)=0, \quad  0 <  x < b^*. \label{H_harmonic}
\end{align}
On the other hand, by the identity (3.19) in \cite{APP2007} and \eqref{laplace_in_terms_of_z}, it follows that, for any $0 < x < b$,
%\begin{equation*}
%{\color{blue}J^{(q,r)}(x)}=\tilde{\E}_x\left(e^{-q\tilde{\tau}_0^-}e^{\Phi(q+r)\tilde{X}(\tilde{\tau}_0^-)};\tilde{\tau}_0^-<\infty\right)+W^{(q)}(x)\frac{r}{\Phi(q+r)-\Phi(q)},
%\end{equation*}
\begin{equation}\label{mart_Z}
	J^{(q,r)}(x)=\tilde{\E}_x\left({\rm e}^{-q\tilde{\tau}_0^-}{\rm e}^{\Phi(q+r)\tilde{X}(\tilde{\tau}_0^-)};\tilde{\tau}_0^-<\infty\right)+ \frac{r W^{(q)}(b) }{\Phi(q+r)-\Phi(q)} \tilde{\E}_x\left( {\rm e}^{-q \tilde{\tau}_b^+}; \tilde{\tau}_0^- > \tilde{\tau}_b^+\right),
\end{equation}
where $\tilde{\mathbb{E}}_x$ is the law of the spectrally negative \lev process $\tilde{X}:=-X$ with $\tilde{X}(0) = x$,  $\tilde{\tau}_0^- := \inf \{ t > 0: \tilde{X}(t) < 0\}$, and $\tilde{\tau}_b^+ := \inf \{ t > 0: \tilde{X}(t) > b\}$. 
%\red{JL: I guess we don't want to introduce this new concept $\mathcal{V}^{(q)}_a$ here?  Instead, we can write
%\begin{equation*}
%{\color{blue}J^{(q,r)}(x)}=\tilde{\E}_x\left(e^{-q\red{\tilde{\tau}_0^-}}e^{\Phi(q+r)\tilde{X}(\red{\tilde{\tau}_0^-})};\red{\tilde{\tau}_0^-}<\infty\right)+W^{(q)}(b) \tilde{\E}_x\left( e^{-q \tilde{\tau}_b^+}; \tilde{\tau}_0^- > \tilde{\tau}_b^+\right) \frac{r}{\red{\Phi(q+r)-\Phi(q)}},
%\end{equation*}
%from which the martingale property is clear?
%}
%The above identity implies that ${\color{blue}J^{(q,r)}(x)}$ lies in the family $\mathcal{V}^{(q)}_a$ defined in Section 2 of \cite{LRZ}. Therefore using Remark 2.1 in \cite{LRZ} we see that the process
%It is clear from \eqref{mart_Z} that
By this and the  strong Markov property (see Section 3.5 of \cite{KKR}), the stopped process
\begin{equation*}
\{{\rm e}^{-q(t\wedge T_{(0,b^*)})}J^{(q,r)}(\tilde{X}(t\wedge T_{(0,b^*)})); t\geq 0\},
\end{equation*}
with $ T_{(0,b^*)}:=\inf\{t>0:\tilde{X} (t)\not\in(0,b^*)\}$ is a martingale. Hence, following the steps in the proof of Theorem 2.1 in \cite{BKY} and noting that $y \mapsto J^{(q,r)}(y)$ on $(0, \infty)$ is sufficiently smooth, we can conclude that 
\begin{align}
(\mathcal{L}-q) J^{(q,r)}(b^*-x)=0, \quad 0 < x < b^*. \label{Z_r_harmonic}
\end{align}
Finally by the proof of Theorem 2.1 in \cite{BKY} we have
\begin{align}\label{Z_q_harmonic}
(\mathcal{L}-q) Z^{(q)}(b^*-x)=0, \quad 0 < x < b^*.
\end{align}
%Now by applying \eqref{H_harmonic}, \eqref{Z_r_harmonic}, and \eqref{Z_q_harmonic} in  \eqref{vf5}, we have the result for $0 < x < b^*$.}

%\red{[JL: How about this? I think it is safer to be explicit that $(\mathcal{L}-q) Z^{(q,r)} (b^* - x) = 0$ to make the referee happy.]  
By  \eqref{Z_r_harmonic} and \eqref{Z_q_harmonic}, we have $(\mathcal{L}-q) Z^{(q,r)} (b^* - x) = 0$.  Applying this and \eqref{H_harmonic} in \eqref{vf5}, we have the result for $0 < x < b^*$.

(ii) For the case $x>b^*$, first we note the following
 \begin{align} \label{generator_parts}
 \begin{split}
 	(\mathcal{L}-q)(b^*-x)&=\psi'(0+)- q(b^*-x)\quad\text{and}\\
 	(\mathcal{L}-q){\rm e}^{\Phi(q+r)(b^*-x)}&=r {\rm e}^{\Phi(q+r)(b^*-x)}.
	\end{split}
 \end{align}
 Now, the equality \eqref{vf5}, for $x > b^*$, can be written, %\red{[cleaned up the eqn below a bit]}
 %using  \eqref{vf_db_a} we can write the value function $v_{b^*}$ in the following form for $x>b^*$
 \begin{align} \label{v_rewritten}
 v_{b^*}(x)=-\frac{r}{r+q}(b^*-x)-\frac{r}{r+q} \Big( \frac {\psi'(0+)} q + \frac{1}{\Phi(q+r)} \Big) -\frac{q}{r+q}\frac{{\rm e}^{\Phi(q+r)(b^*-x)}}{\Phi(q+r)}.%\qquad\text{for $x>b^*$.}
 \end{align}
 Therefore we have for $x>b^*$, by using \eqref{generator_parts} and \eqref{v_rewritten},% \red{[remove the middle equality?  I also cleaned up the last equality a bit.]}
 \begin{align*}
 	(\mathcal{L}-q)v_{b^*}(x)%&=-\frac{r}{r+q}\psi'(0+)+\frac{qr}{r+q}(b^*-x)+\frac{r}{r+q}\psi'(0+)+\frac{qr}{r+q}\frac{1}{\Phi(q+r)}\\
 	%&-\frac{qr}{r+q}\frac{e^{\Phi(q+r)(b^*-x)}}{\Phi(q+r)}\\
 	%&
 	=\frac{qr}{r+q} \Big( (b^*-x)+\frac{1-{\rm e}^{\Phi(q+r)(b^*-x)}}{\Phi(q+r)} \Big).
 \end{align*}
%And the result follows.

%\red{[added this]}
Suppose $b^* = 0$. By \eqref{value_function_zero}, % \red{[cleaned up the last eqn a bit]}
\begin{align*}
(\mathcal{L}-q) v_0(x)
&= \frac r {r+q}  (\mathcal{L}-q) \Big[ x -  \Big(\frac {\psi'(0+)} {r+q} - \rho \Big)   +  \Big(\frac {\psi'(0+)} {r+q} + \frac {\rho q} r\Big) {\rm e}^{- \Phi(q+r)x}\Big) \Big]\\
&= \frac r {r+q}  \Big[ -\psi'(0+) - qx  +  q \Big(\frac {\psi'(0+)} {r+q} - \rho \Big) + \Big(\frac {\psi'(0+)} {r+q} + \frac {\rho q} r\Big) r {\rm e}^{-\Phi(q+r)x}\Big],
%&= \frac r {r+q}  \Big[ -\psi'(0+) - qx  +  \frac {\psi'(0+)} {r+q} \Big( q + r e^{-\Phi(q+r)x} \Big) \Big],
%&= \frac r {r+q}  \Big[ - qx  - r \frac {\psi'(0+)} {r+q} + \frac {r} {r+q} \psi'(0+)e^{-\Phi(q+r)x}\Big].
\end{align*}
which reduces to \eqref{generator_zero_case}.
\end{proof}

Now we have all the elements to prove Theorem \ref{ver_theo}.
\begin{proof}[Proof of Theorem \ref{ver_theo}]
By Lemmas \ref{smooth_fit_prob1} and  \ref{verificationlemma}, it is sufficient to show \eqref{HJB-inequality} with $v_{\hat{\pi}}$ replaced with $v_{b^*}$. In particular, we show that \eqref{HJB-inequality} holds with equality.

 % -- and insert the proof currently in Lemma \ref{optimality_comp}.
%\end{proof}
%}

%Now we have the elements to prove the conditions of the HJB inequality associated to the problem given by \eqref{HJB-inequality}.
%\begin{lemma}\label{optimality_comp}
%For $x>0$ we have that $(\mathcal{L}-q)v_{b^*}(x)+r\max_{0\leq l\leq x}(l+v_{b^*}(x-l)-v_{b^*}(x))=0$. %\red{Is this $=$? or inequality?}.
%\end{lemma}
%\begin{proof}
(i) Suppose $b^* > 0$. For $x\leq b^*$, Lemmas  \ref{cond_max_v} and  \ref{generator_on_v} show \eqref{HJB-inequality} with equality.
%\[
%(\mathcal{L}-q)v_{b^*}(x)+r\max_{0\leq l\leq x}(l+v_{b^*}(x-l)-v_{b^*}(x))=0.
%\]
Now for the case $x > b^*$, by substituting \eqref{vf5} in \eqref{max_cond}, %\red{[remove the coefficient $r$ on both sides for simplicity]} 
%\red{[simplified below a bit]}
\begin{align*}
\max_{0\leq l\leq x} \{ l+v_{b^*}(x-l)-v_{b^*}(x) \}% = x-b^*+v_{b^*}(b^*)-v_{b^*}(x) \\
%&=(x-b^*)+\left(-\frac{r}{r+q}\frac{\psi'(0+)}{q}-\frac{1}{\Phi(q+r)}\right)\\
%&+\frac{r}{r+q}(b^*-x)+\psi'(0+)\frac{r}{q(r+q)}+\frac{r}{r+q}\frac{1}{\Phi(q+r)}+\frac{q}{r+q}\frac{1}{\Phi(q+r)}e^{\Phi(q+r)(b^*-x)}\\
%&
=- \frac{q}{r+q} \Big[ (b^*-x)+ \frac{1- {\rm e}^{\Phi(q+r)(b^*-x)}}{\Phi(q+r)} \Big].
\end{align*}
%\red{[remove the middle equality?]}
%Hence putting the pieces together we obtain that
This together with Lemma \ref{generator_on_v} shows \eqref{HJB-inequality} with equality.
%\[
%(\mathcal{L}-q)v_{b^*}(x)+r\max_{0\leq l\leq x}(l+v_{b^*}(x-l)-v_{b^*}(x))=0.
%\]

%\red{
(ii) Suppose now $b^* = 0$.  Again, by substituting \eqref{value_function_zero} in \eqref{max_cond},
% the slope condition, %[below: initial computations (especially the signs) are very wrong]
\begin{align*}
\max_{0\leq l\leq x} \{ l+v_{b^*}(x-l)-v_{b^*}(x) \} &=
x+v_{0}(0)-v_{0}(x) \\
&= x + \rho - \frac r {r+q} \Big[ x -  \Big(\frac {\psi'(0+)} {r+q} - \rho \Big)   +  \Big(\frac {\psi'(0+)} {r+q} + \frac {\rho q} r\Big) {\rm e}^{- \Phi(q+r)x}\Big) \Big]
%= x -  \frac r {r+q} \Big[ x -  \frac {\psi'(0+)} {r+q}  \Big( 1 -  e^{- \Phi(q+r)x}\Big) \Big]  
\\&= \frac q {r+q}  x + \frac r {r+q} \Big(\frac {\psi'(0+)} {r+q} + \frac {\rho q} r\Big) \Big( 1 -  {\rm e}^{- \Phi(q+r)x}\Big). 
\end{align*}
Combining this and \eqref{generator_zero_case}, we have \eqref{HJB-inequality} with equality, as desired.
%
%\begin{align*}
%&(\mathcal{L}-q) v_0(x) + r \max_{0\leq l\leq x}(l+v_{\red{0}}(x-l)-v_{\red{0}}(x)) 
%% \\
%%&=  \frac r {r+q}  \Big[ - qx  - r \frac {\psi'(0+)} {r+q} + \frac {r} {r+q} \red{\psi'(0+)} e^{-\Phi(q+r)x}\Big] + r \Big[\frac q {r+q}  x +  r \frac {\psi'(0+)} {(r+q)^2}  \Big( 1 -  e^{- \Phi(q+r)x}\Big) \Big]
% = 0,
%\end{align*}%}
%%And the result is proved.
%as desired.
 \end{proof}
%	\red{[JL: So, I am fuzzy about the Ito's formula. But here, we only showed that $v^{b^*}$ is sufficently smooth on $(0, \infty)$. But I guess it needs some smoothness at $\{0\}$ as well? If so, do we need to extend the function to $\R$; the proof still holds by using the Meyer-Ito type, because it is differentiable at $0$ by the Remark below.]}

%\red{[added this remark]}
\begin{remark}[Connection with the classical case] \label{remark_convergence_r}
Suppose $\rho=0$. It is expected that as $r \rightarrow \infty$ the optimal barrier $b^*$ as well as the value function $v_{b^*}$ converge to those in the classical case (assuming $\E X_1 = -\psi'(0+) > 0$): $\tilde{b}^* := (\overline{Z}^{(q)})^{-1} ( -\psi'(0+) / q )$ and
%\begin{equation}\label{eq:opt-threshold}
%a^* = \begin{cases} \left(\overline{Z}^{(q)}\right)^{-1}\left (\frac {-\psi'(0+)} q \right)> 0 & \text{if $\mu>0$}, \\
%0   & \text{if $\mu\leq 0$}. \end{cases}
%\end{equation}
\begin{equation}\label{eq:va-star}
\tilde{v}(x) :=
 - \overline{Z}^{(q)} (\tilde{b}^* -x) - \frac {\psi'(0+)} q, \quad x \geq 0, \end{equation}
as obtained in Bayraktar et al.\ \cite{BKY}.  

This can be easily confirmed as follows.  First, it is easy to see that the function $f(b)$ as in \eqref{def_f} converges, as $r \rightarrow \infty$, to $\overline{Z}^{(q)}(b)+ {\psi'(0+)} /q$, whose root becomes $\tilde{b}^*$. Moreover, in view of the form of the value function \eqref{vf5},
\begin{align}
-H^{(q,r)}(b-x)-\frac{Z^{(q,r)}(b-x)}{\Phi(q+r)}  \xrightarrow{r \rightarrow \infty}- \overline{Z}^{(q)} (b -x) - \frac {\psi'(0+)} q. \label{value_function_conv}
\end{align}
This is a rough illustration of how the convergence holds.  In Section \ref{section_numerics}, we numerically verify the convergence.
\end{remark}

\section{Solutions to the extension with classical bail-outs} \label{section_second_prob}
In this section, we solve the second problem as  defined in Section \ref{second_prob}.  To this end, we take essentially the same steps as in the previous section: first choosing the candidate barrier $b^\dagger$ using the smoothness conditions and then showing that $u_{b^\dagger}$ solves the required variational inequalities.  Because of the similarities of the forms of $v_b$ (for $\rho = 0$) and $u_b$ as in \eqref{vf3} and \eqref{vf_db_a}, the computation will be similar.  In addition, it turns out that the value function will have the same form (with the different barrier) as that in the first problem and hence many of the results in Section \ref{section_first_prob} can be reused. 

\subsection{Smooth fit and selection of the candidate barrier $b^\dagger$}
We first note that the expected NPV given in \eqref{vf_db_a}  can be obtained from \eqref{vf3} for $\rho = 0$, just by replacing $H^{(q,r)}(b)/Z^{(q,r)}(b)$ with $[{r Z^{(q)} (b)} / (r+q)-\beta  ] /Z^{(q,r)\prime}(b)$. Hence the analysis follows verbatim from Section \ref{sf_d} and we obtain the following lemma simply by modifying Lemma \ref{smooth_fit_prob1}.
\begin{lemma} \label{lemma_smoothness_capital_injection}Suppose $b > 0$ is such that the condition
\begin{equation}\label{opt_thres_bailout}
\widehat{\mathfrak{C}}_b: -\frac{1}{Z^{(q,r) \prime}(b)}\left(\frac{r Z^{(q)} (b)}{ r+q}-\beta\right)=\frac{1}{\Phi(q+r)}
\end{equation}
 is satisfied. Then, $u_{b}$ is $C^2 (0, \infty)$ for the case $X$ is of bounded variation, while it is $C^3 (0, \infty)$ for the case $X$ is of unbounded variation.
\end{lemma} 

We shall now show the existence of $b$ such that \eqref{opt_thres_bailout} holds.  Differently from the first problem (see Lemma \ref{lemma_opt_thres}), such $b$ exists all the time thanks to the assumption that $\beta > 1$.
\begin{lemma} \label{lemma_existence_b_dagger}
	There exists a unique solution $\tilde{b}>0$ to the equation \eqref{opt_thres_bailout}. %\red{JL: Similarly to Lemma \ref{lemma_opt_thres}, use $\tilde{b}$ only in the lemma?}
\end{lemma}
\begin{proof} %\red{[changed from $f$ to $\hat{f}$]}
%The proof follows by noting that if $b^\dagger$ is the unique root of $f$, then it satisfies 
First notice, by \eqref{Z_q_r_der}, that \eqref{opt_thres_bailout} is equivalent to $\hat{f}(b) = 0$ where
	\begin{align}
	\hat{f}(b):=
	%\frac{r}{r+q}Z^{(q)}(b)+\frac{q}{r+q}{\color{blue}J^{(q,r)}(b)}
	Z^{(q,r)}(b) -\beta, \quad b \geq 0. \label{def_f_hat}
	\end{align}
	Differentiating this and again by \eqref{Z_q_r_der}, %\red{[remove the middle equality below?]}
	\begin{align*}
		\hat{f}'(b)%&=\frac{rq}{r+q}W^{(q)}(b)+\frac{q}{r+q} \big[ \Phi(q+r){\color{blue}J^{(q,r)}(b)}-rW^{(q)}(b) \big] \\
		%&
		=\frac{q}{r+q}\Phi(q+r) J^{(q,r)}(b)>0.
	\end{align*}
	Therefore the function $\hat{f}$ is strictly increasing, and we note that, by \eqref{W_q_limit}, $\lim_{b\to\infty}\hat{f}(b)=\infty$.
	On the other hand, $\beta>1$ implies that $\hat{f}(0)=1-\beta<0$.
Hence, there exists a unique $\tilde{b}>0$ such that $\hat{f} (\tilde{b})=0$, as desired.
	\end{proof}
\subsection{Verification.}
Let $b^{\dagger} > 0$ be the unique root of \eqref{opt_thres_bailout} as in Lemma \ref{lemma_existence_b_dagger}. By substituting \eqref{opt_thres_bailout} in \eqref{vf_db_a}, we can write
%\begin{equation}\label{opt_thres_db}
%\frac{1}{Z^{(q,r) \prime}(b)}\left(\frac{r Z^{(q)} (b)}{ (r+q)}-\beta\right)=-\frac{1}{\Phi(q+r)}.
%\end{equation}
%So using \eqref{opt_thres} we can write the value function in the following form
\begin{align}\label{vf_db_c}
	u_{b^\dagger}(x)=-H^{(q,r)}(b^\dagger-x)-\frac{Z^{(q,r)}(b^\dagger-x)}{\Phi(q+r)}\qquad\text{for $x \geq0$}.
\end{align}	

\begin{remark} \label{remark_same_form}
The function \eqref{vf_db_c} has the same form as the value function for the first problem \eqref{vf5}, except that the value of the barrier is different.
\end{remark}

%With $b^{\dagger} > 0$ defined above, w
%We shall show the optimality of the obtained periodic barrier strategy %reflection strategy at Poissonian dividend decision times 
%with additional classical bail-outs $\bar{\pi}^{b^{\dagger}}$.

The main result of this section is as follows.
\par 
%And for $x>b$
%\begin{align}\label{vf_db_d}
%	u_{{\color{blue}b^\dagger}}(x)=-r\frac{({\color{blue}b^\dagger}-x)(q+r)+(1-e^{\Phi(q+r)({\color{blue}b^\dagger}-x)})\psi'(0+)}{(q+r)^2}+\left(\frac{r}{(r+q)}+\frac{q}{(r+q)}e^{\Phi(q+r)({\color{blue}b^\dagger}-x)}\right)u_{{\color{blue}b^\dagger}}({\color{blue}b^\dagger}),
%\end{align}
%with
%\begin{equation}\label{opt_vf_bail}
%u_{{\color{blue}b^\dagger}}(b)=-\frac{r}{(r+q)}\frac{\psi'(0+)}{q}-\frac{1}{\Phi(q+r)}.
%\end{equation}
\begin{theorem} \label{proof_main2}
The periodic barrier strategy %reflection strategy 
$\bar{\pi}^{b^{\dagger}}$ %at Poissonian times with reflection 
%at ${b^\dagger}$ and
with classical reflection from below at $0$
is optimal and the value function is $u(x)=u_{{b^\dagger}}(x)$ for all $0\leq x<\infty$.
\end{theorem}
As in Section \ref{ver_div} (for the  proof of Theorem \ref{ver_theo}), we shall provide the verification lemma and then show that $u_{{b^\dagger}}$ satisfies the stated conditions.
To this end, we extend the domain of the function $u_{\bar{\pi}}$, for all $\bar{\pi} \in \bar{\mathcal{A}}$, as in \eqref{vf}, to all $\mathbb{R}$ by setting $u_{\bar{\pi}}(x):=u_{\bar{\pi}}(0)+\beta x$ for $x < 0$.
\par A sufficient condition for optimality is given as follows.
\begin{lemma}[Verification lemma]
	\label{verificationlemma_2}
	Suppose $\hat{\pi}$ is an admissible dividend strategy such that $u_{\hat{\pi}}$ is sufficiently smooth on $(0,\infty)$ and differentiable at zero (i.e.\ $u_{\hat{\pi}}'(0) = \beta$), and satisfies
	\begin{align}
		\label{HJB-inequality_db}
		(\mathcal{L} - q)u_{\hat{\pi}}(x)+r\max_{0\leq l\leq x}\{l+u_{\hat{\pi}}(x-l)-u_{\hat{\pi}}(x)\}\leq 0,  &\quad x > 0,   \\
		u_{\hat{\pi}}'(x)\leq \beta,  &\quad x > 0,  \label{HJB-inequality_db2} \\
		\inf_{x\geq0}u_{\hat{\pi}}(x)>-m, &\quad \text{for some $m > 0$}. \label{HJB-inequality_db3}
	\end{align} %\red{changed to $m$ because it would be confusing with the martingale $M$.}
	Then $u_{\hat{\pi}}(x)=u(x)$ for all $x\geq0$ and hence $\hat{\pi}$ is an optimal strategy.
\end{lemma}
\begin{proof}
See Appendix \ref{Appendix_B}.
\end{proof}

%\red{[I rewrote this way, is this ok?] 
With the help of our arguments for the verification in the previous section, we shall show that $u_{b^\dagger}$ satisfies the inequalities \eqref{HJB-inequality_db}, \eqref{HJB-inequality_db2}, and \eqref{HJB-inequality_db3}. Recall that $u_{b^\dagger}$ is sufficiently smooth  on $(0, \infty)$ by Lemma \ref{lemma_smoothness_capital_injection}.
\begin{proof}[Proof of Theorem \ref{proof_main2}]
By the observation given in Remark \ref{remark_same_form}, the following analogues of Lemmas \ref{cond_b^*} (i), \ref{cond_max_v}, and \ref{generator_on_v} hold in the same way. %Hence we have the following.
\begin{lemma}\label{verificationlemma_B}
The following conditions hold for $u_{{b^\dagger}}$: 
	\begin{itemize}
		\item[(i)]The value function $u_{{b^\dagger}}$ is strictly increasing, concave, and $u_{{b^\dagger}}'({b^\dagger})=1$.
	\item[(ii)] The inequality \eqref{HJB-inequality_db} for $x > 0$ holds with $u_{\hat{\pi}}$ replaced with $u_{b^\dagger}$.
%	\item[(iii)] For $0<x<{b^\dagger}$ we have that $(\mathcal{L}-q)u_{{b^\dagger}}(x)=0$. \red{[do we need (iii)?]}
	\end{itemize}
	\end{lemma}
	
Hence, we are only left to show \eqref{HJB-inequality_db2} (with the differentiability at $0$) and \eqref{HJB-inequality_db3}. For the former, by Lemma \ref{verificationlemma_B} (i), it is sufficient to show $u_{{b^\dagger}}'(0) = \beta$. This indeed holds because, by \eqref{opt_thres_bailout} and \eqref{vf_db_c}, we obtain
\begin{align*}
u_{{b^\dagger}}'(0)=\frac{r Z^{(q)} (b^\dagger)}{ r+q}+\frac{Z^{(q,r)\prime}(b^\dagger)}{\Phi(q+r)}=\frac{r Z^{(q)} (b^\dagger)}{ r+q}-\left(\frac{r Z^{(q)} (b^\dagger)}{ r+q}-\beta\right)=\beta.
%\frac{r}{r+q}Z^{(q)}({\color{blue}b^\dagger})-\frac{1}{\Phi(q+r)}Z^{(q,r)\prime}({\color{blue}b^\dagger})=\frac{r}{r+q}Z^{(q)}({\color{blue}b^\dagger})+({\color{blue}H^{(q,r)\prime}}({\color{blue}b^\dagger})-\beta)=\beta,
\end{align*}
%where in the second equality we used \eqref{opt_thres_bailout}.
%Now in order to prove (v) we just note that the function $u_{{\color{blue}b^\dagger}}$ is increasing, therefore we obtain that
For the latter, it holds by Lemma \ref{verificationlemma_B} (i) and because $u_{b^\dagger}(0)$ is finite. %\red{[simplified this way, ok?]}
%\begin{align*}
%u_{{b^\dagger}}(x)\geq u_{{b^\dagger}}(0)=-\left(H^{(q,r)}({b^\dagger})+\frac{Z^{(q,r)}({b^\dagger})}{\Phi(q+r)}\right).
%\end{align*}
%Hence \eqref{HJB-inequality_db3} follows by setting $m:=H^{(q,r)}({b^\dagger})+Z^{(q,r)}(b^\dagger) / \Phi(q+r) \in \R$.
\end{proof}

%\red{[added this remark]}
\begin{remark}[Connection with the classical case] \label{remark_convergence_r_capital_injection} Similarly to Remark \ref{remark_convergence_r}, as $r \rightarrow \infty$, the optimal barrier $b^\dagger$ as well as the value function $u_{b^\dagger}$ are expected to converge to those in the classical case: $\tilde{b}^\dagger :=  (Z^{(q)})^{-1}(\beta)$ and 
\begin{equation} 
\tilde{u}(x) :=  - \overline{Z}^{(q)}(\tilde{b}^\dagger-x) - \frac {\psi'(0+)} q, \quad x \geq 0;\label{v_bar_a_bailout}
\end{equation}
see Bayraktar et al.\ \cite{BKY}. 

Indeed, it is easy to see that the function $\hat{f}(b)$ as in \eqref{def_f_hat} converges, as $r \rightarrow \infty$, to $Z^{(q)}(b) - \beta$, whose root becomes $\tilde{b}^\dagger$. The form of the value function also converges to that of \eqref{v_bar_a_bailout} as in \eqref{value_function_conv}.  This is numerically verified in Section \ref{section_numerics}.
\end{remark}

\section{Numerical Examples} \label{section_numerics}

%We have
%\begin{align*}
%Z^{(q)}(x) = 1+ q \frac {e^{\Phi(q) x}-1} {\Phi(q) \psi'(\Phi(q))} + q \sum_{i \in \mathcal{I}_q} \frac {C_{i,q}} {\xi_{i,q}} (e^{-\xi_{i,q}x}-1)
%\end{align*}
%
%We have
%\begin{align*}
%\int_{0}^{x} W^{(q)}(y) e^{-\Phi(q+r) y} \ud y &= \int_{0}^{x} \Big[ \frac {e^{\Phi(q) y}} {\psi'(\Phi(q))} - \sum_{i \in \mathcal{I}_q} C_{i,q} e^{-\xi_{i,q}y}  \Big] e^{-\Phi(q+r) y} \ud y \\
%&= \int_{0}^{x} \Big[ \frac {e^{(\Phi(q)-\Phi(q+r)) y}} {\psi'(\Phi(q))} - \sum_{i \in \mathcal{I}_q} C_{i,q} e^{-(\xi_{i,q}+\Phi(q+r))y}  \Big] \ud y \\
%&=  \frac {e^{(\Phi(q)-\Phi(q+r)) x}-1} {(\Phi(q)-\Phi(q+r))\psi'(\Phi(q))} + \sum_{i \in \mathcal{I}_q} \frac {C_{i,q}} {\xi_{i,q}+\Phi(q+r)} (e^{-(\xi_{i,q}+\Phi(q+r))x}-1)  
%\end{align*}

In this section, we confirm the analytical results obtained for the two problems through a sequence of numerical experiments.  Throughout this section, we assume that the underlying process $X$ is the spectrally positive version of the \emph{phase-type} \lev process (with a Brownian motion) of \cite{Asmussen_2004}, which admits an analytical form of scale function as in \cite{Egami_Yamazaki_2010_2}.  This process is particularly important because it can approximate any  spectrally positive \lev process (see \cite{Asmussen_2004} and \cite{Egami_Yamazaki_2010_2}).  See, e.g., \cite{Asmussen_2004, LYZ} for stochastic control problems using this process.

More specifically, for some $c \in \R$ and $\sigma > 0$, % be a spectrally positive process with i.i.d.\ phase-type distributed jumps \cite{Asmussen_2004} of the form
\begin{equation}
 X(t) - X(0)= - c t+\sigma B(t) + \sum_{n=1}^{N(t)} Z_n, \quad 0\le t <\infty, \label{X_phase_type}
\end{equation}
where $B=( B(t); t\ge 0)$ is a standard Brownian motion, $N=(N(t); t\ge 0 )$ is a Poisson process with arrival rate $\kappa$, and  $Z = ( Z_n; n = 1,2,\ldots )$ is an i.i.d.\ sequence of phase-type-distributed random variables with representation $(m,{\bm \alpha},{\bm T})$, or equivalently the first absorption time in a continuous-time Markov chain consisting of a single absorbing state and $m$ transient states with its initial distribution $\alpha$ and transition matrix ${\bm T}$ (see \cite{Asmussen_2004} for details). The processes $B$, $N$, and $Z$ are assumed mutually independent.  We refer the reader to \cite{Egami_Yamazaki_2010_2, KKR} for the forms of the corresponding scale functions.

\subsection{Numerical results for the first problem}  We first consider the first problem defined in Section \ref{dividends-strategy} and confirm the results obtained in Section \ref{section_first_prob}.  Here, for $X$ in \eqref{X_phase_type}, we set $\sigma = 0.2$ and $\kappa = 2$ and, for $Z$, we use the phase-type distribution with  $m=6$ 
that gives an approximation to the (folded) normal random variable with mean $0$ and variance $1$, which is given in \cite{leung2015analytic} (see \cite{leung2015analytic} for the values of $\alpha$ and ${\bm T}$).  For the drift parameter $c$, we consider \textbf{Case 1} with $c=0.5$  and \textbf{Case 2} with $c=2.0$ to obtain the cases $b^* > 0$ and $b^* = 0$, respectively.  For the other parameters, let $q = 0.05$, $r = 0.1$ and $\rho = 0$ unless stated otherwise.

The first step in the implementation is to compute the optimal barrier $b^*$.  In the left column of Figure \ref{problem1_optimality}, we plot the function $f$ as in \eqref{def_f} for \textbf{Cases 1} and \textbf{2}.
In both cases, it can be confirmed that $f$ is monotonically increasing.  In \textbf{Case 1}, it starts at a negative value and hence its unique root becomes $b^* > 0$; in \textbf{Case 2}, $f$ is uniformly positive and hence $b^* =0$.  In the right column, we plot the corresponding value functions $v_{b^*}$ (solid) along with suboptimal NPVs $v_{b}$ given in \eqref{vf3} (dotted) with $b \neq b^*$.  It can be confirmed in both cases that $v_{b^*}$ dominates $v_{b}$,  for $b \neq b^*$,
 uniformly in $x$.

\begin{figure}[htbp]
\begin{center}
\begin{minipage}{1.0\textwidth}
\centering
\begin{tabular}{cc}
 \includegraphics[scale=0.35]{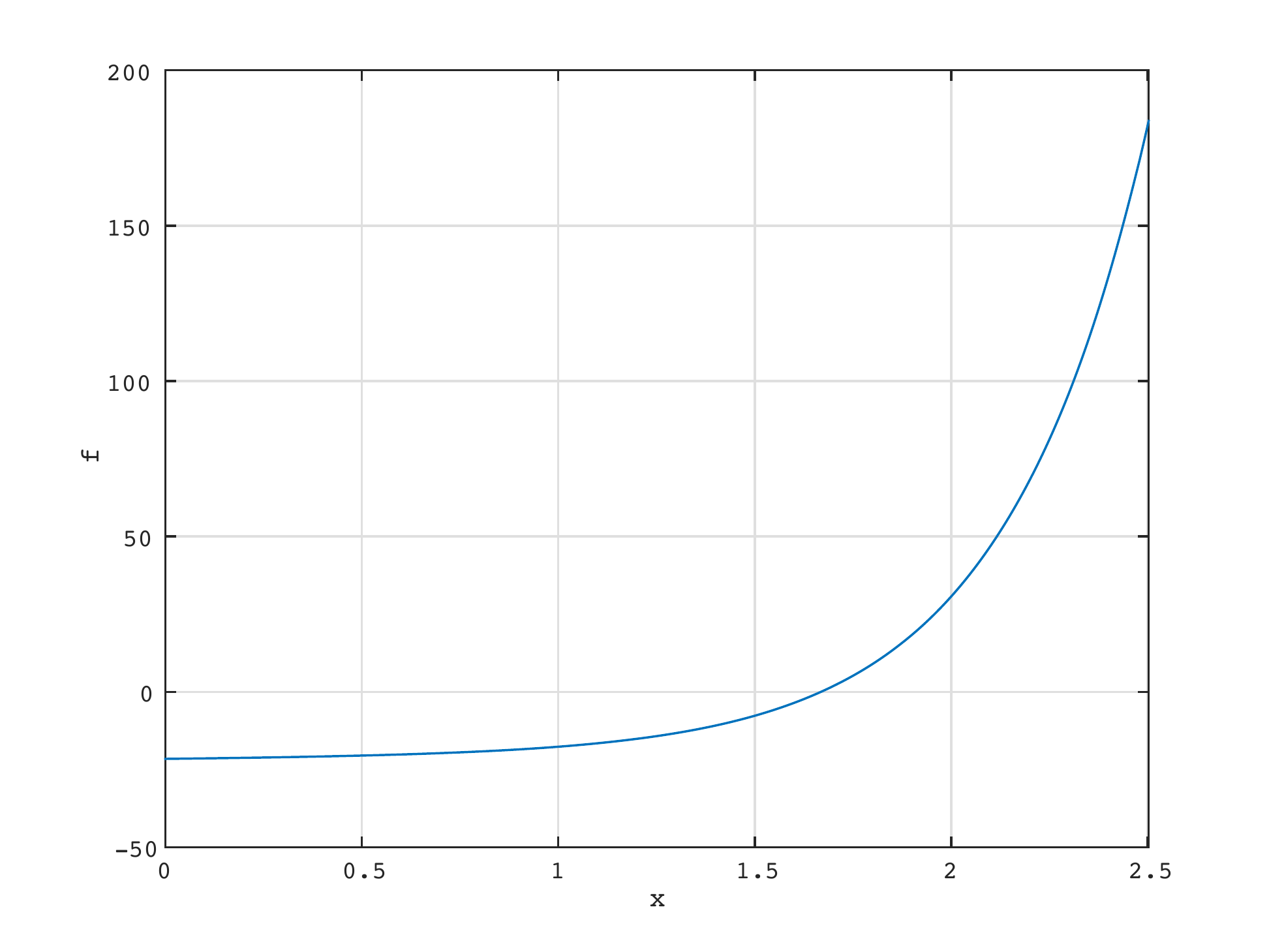} & \includegraphics[scale=0.35]{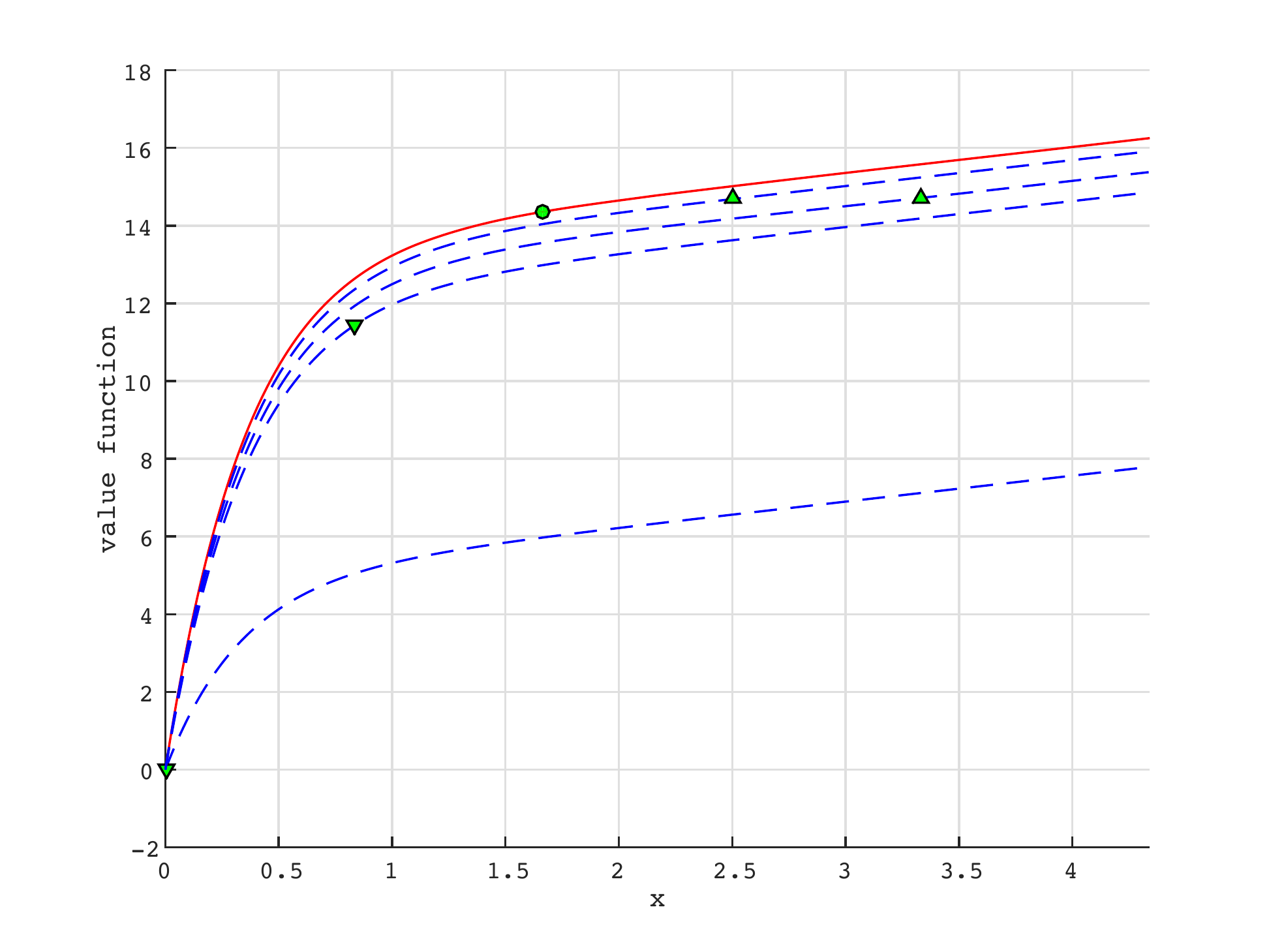}  \\
  \includegraphics[scale=0.35]{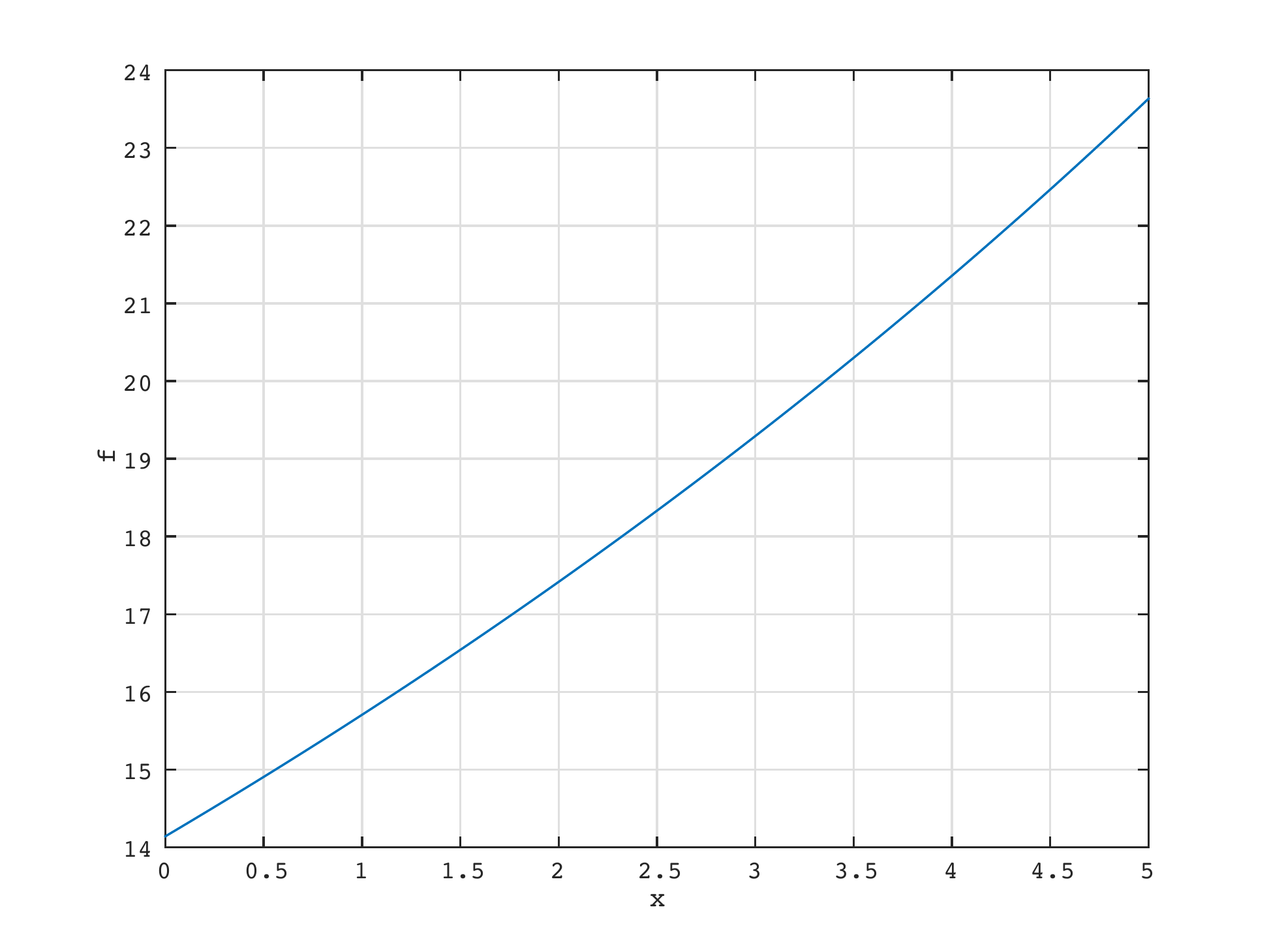} & \includegraphics[scale=0.35]{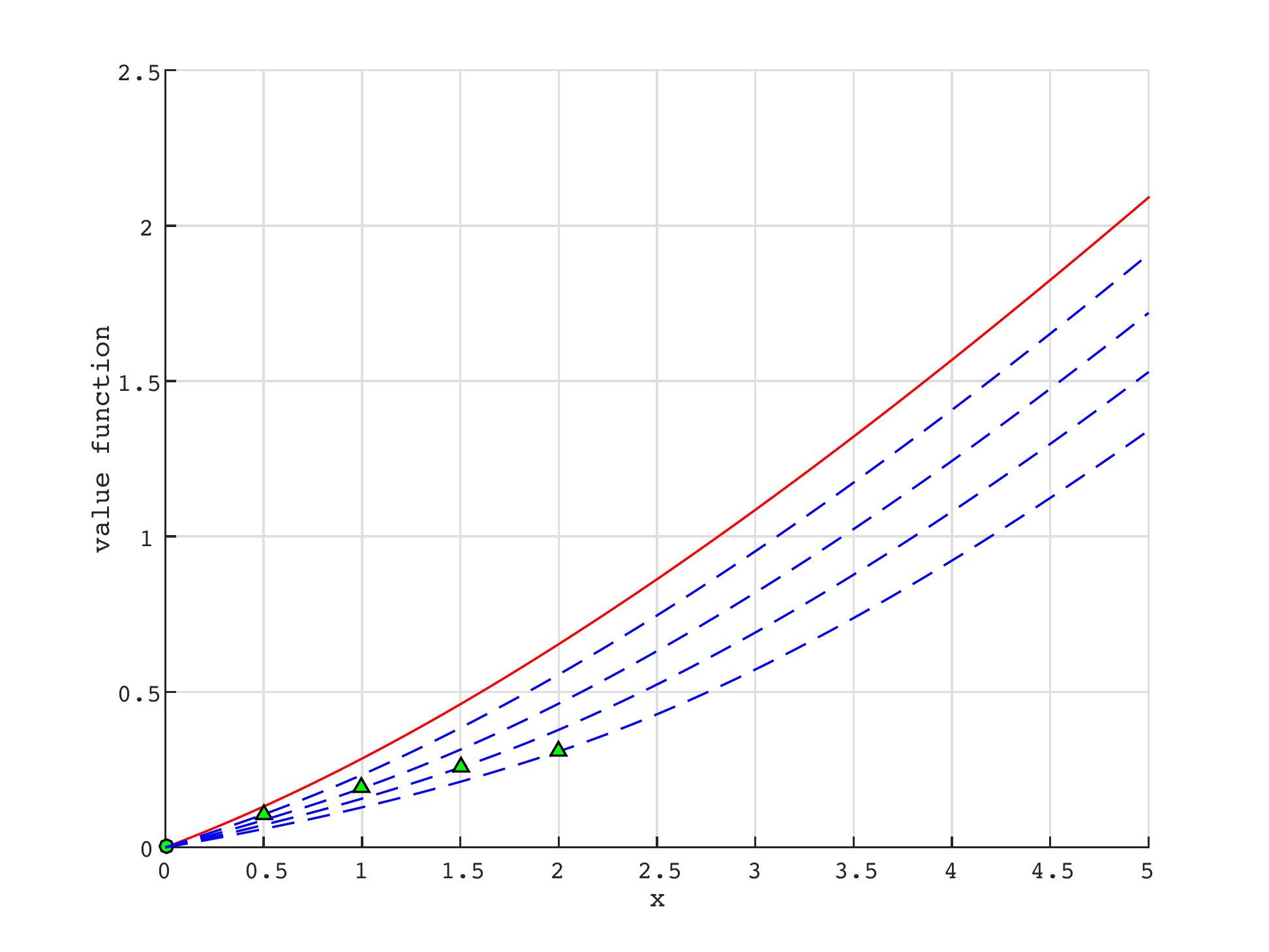}  \\
 $b \mapsto f(b)$ & $x \mapsto v_{b^*}(x)$ and $x \mapsto v_b(x)$ \end{tabular}
\end{minipage}
\caption{(Left) Plots of $f$ as in \eqref{def_f}  for \textbf{Case 1} (top) and \textbf{Case 2} (bottom). Its root, if it exists, becomes $b^* > 0$; otherwise $b^* = 0$. (Right) The corresponding value function $v_{b^*}$ (solid) along with suboptimal expected NPVs $v_b$ (dotted) for $b = 0, b^*/2, 3 b^*/2, 2 b^*$ for \textbf{Case 1} and $b = 0.5, 1, 1.5, 2$ for \textbf{Case 2}.  The values at $b^*$ are indicated by circles whereas those at the suboptimal barriers $b > b^*$ (resp.\ $b < b^*$)  are indicated by up-pointing (resp.\ down-pointing) triangles.
} \label{problem1_optimality}
\end{center}
\end{figure}

We next study the behavior of the value function $v_{b^*}$ with respect to $r$,
% (the expected number of dividend-decision opportunities in a unit time), 
 as we have discussed in Remark \ref{remark_convergence_r}. 
% As $r$ increases, the frequency of dividend-decision times increases.
% It is conjectured that as $r \rightarrow \infty$ the optimal barrier $b^*$ as well as the value functions $v_{b^*}$ converge to those in the classical case (assuming \red{$\E X_1 = -\psi'(0+) > 0$}): $\tilde{b}^* := (\overline{Z}^{(q)})^{-1} ( -\psi'(0+) / q )$ and
%%\begin{equation}\label{eq:opt-threshold}
%%a^* = \begin{cases} \left(\overline{Z}^{(q)}\right)^{-1}\left (\frac {-\psi'(0+)} q \right)> 0 & \text{if $\mu>0$}, \\
%%0   & \text{if $\mu\leq 0$}. \end{cases}
%%\end{equation}
%\begin{equation}\label{eq:va-star}
%\tilde{v}(x) :=
% - \overline{Z}^{(q)} (\tilde{b}^* -x) - \frac {\psi'(0+)} q, \quad x \geq 0, \end{equation}
%as obtained in Bayraktar et al.\ \cite{BKY}.  
Here we use the same parameters as \textbf{Case 1} above except for $r$. In Figure \ref{fig_rho} (i), we plot $v_{b^*}$ for an increasing sequence of $r$ along with the value function $\tilde{v}$ given in \eqref{eq:va-star} in the classical case. It is confirmed that $v_{b^*}$ increases uniformly in $x$ to $\tilde{v}$.  The convergence of the optimal barrier $b^*$ to $\tilde{b}^*$ is also confirmed.

%\begin{figure}[htbp]
%\begin{center}
%\begin{minipage}{1.0\textwidth}
%\centering
%\begin{tabular}{c}
% \includegraphics[scale=0.4]{fig_conv_r}  \end{tabular}
%\end{minipage}
%\caption{The value functions $v_{b^*}$ (dotted) for $r = 0.01$, $0.02$, $0.05$,  $0.1$, $0.2$, $0.5$, $1$, $2$, $3$, $4$, $5$ along with the value function $\tilde{v}$ (given in \eqref{eq:va-star}) in the classical case (solid).  The up-pointing triangles show the points at $b^*$ of $v_{b^*}$; the circle shows the point at $\tilde{b}^*$ of $\tilde{v}$.
%} \label{convergence_r}
%\end{center}
%\end{figure}

Finally, we study the behavior of the value function $v_{b^*}$ with respect to the terminal payoff at ruin $\rho$.
Here we use the same parameters as \textbf{Case 1} above except for $\rho$. In Figure \ref{fig_rho} (ii), we plot $v_{b^*}$ for an increasing sequence of $\rho$ ranging from $-20$ to $20$. It is confirmed that, as $\rho$ increases, $v_{b^*}$ increases  uniformly in $x$ while $b^*$ decreases.  For sufficiently large $\rho$, $b^*$ becomes $0$. Interestingly, for high enough $\rho$, the value function fails to be monotonically increasing; this is due to the fact that, while one wants to liquidate as quickly as possible to enjoy the terminal payoff at ruin, one must wait until the next dividend payment opportunity.
\begin{figure}[htbp]
\begin{center}
\begin{minipage}{1.0\textwidth}
\centering
\begin{tabular}{cc}
\includegraphics[scale=0.35]{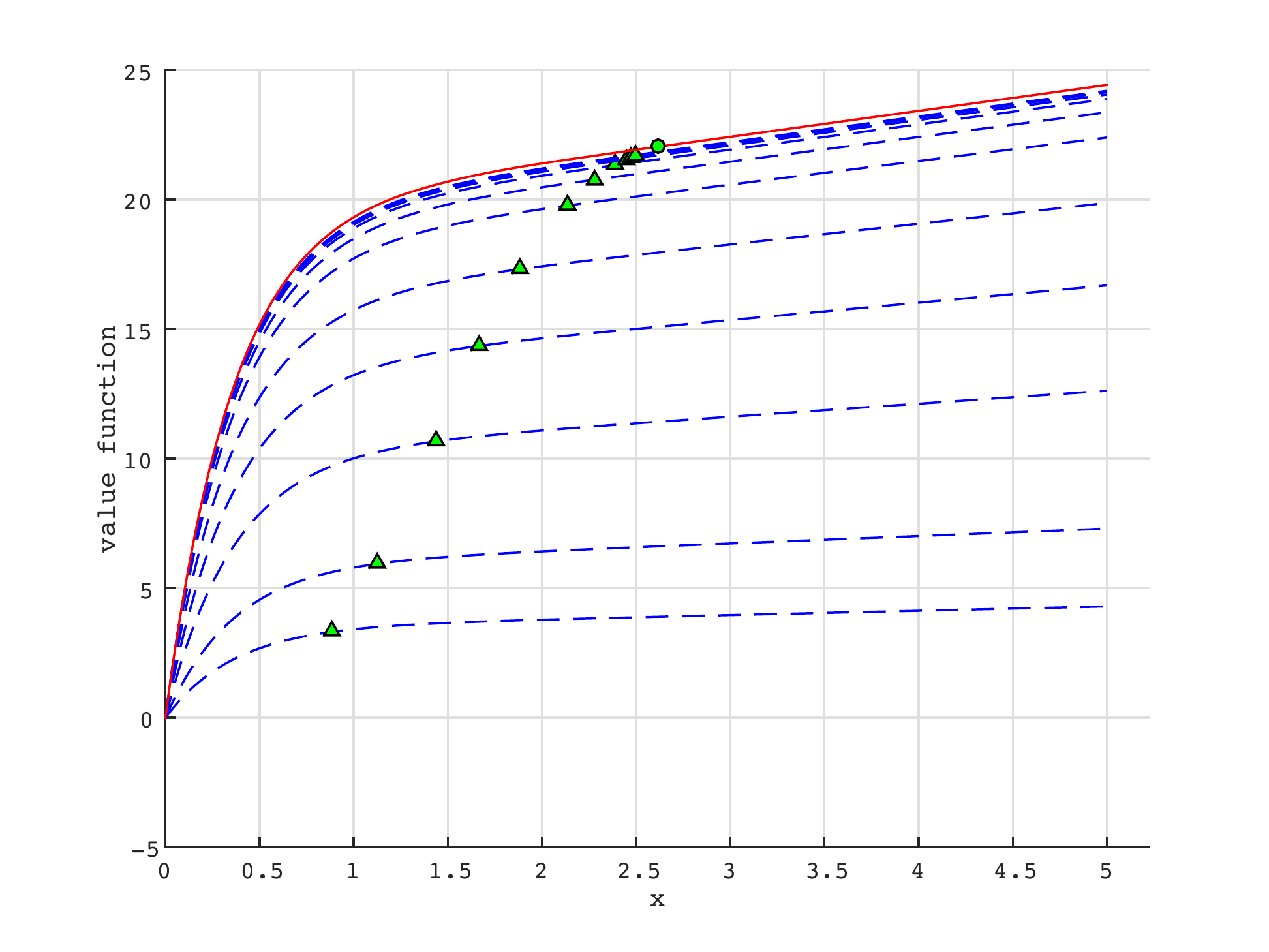} & \includegraphics[scale=0.35]{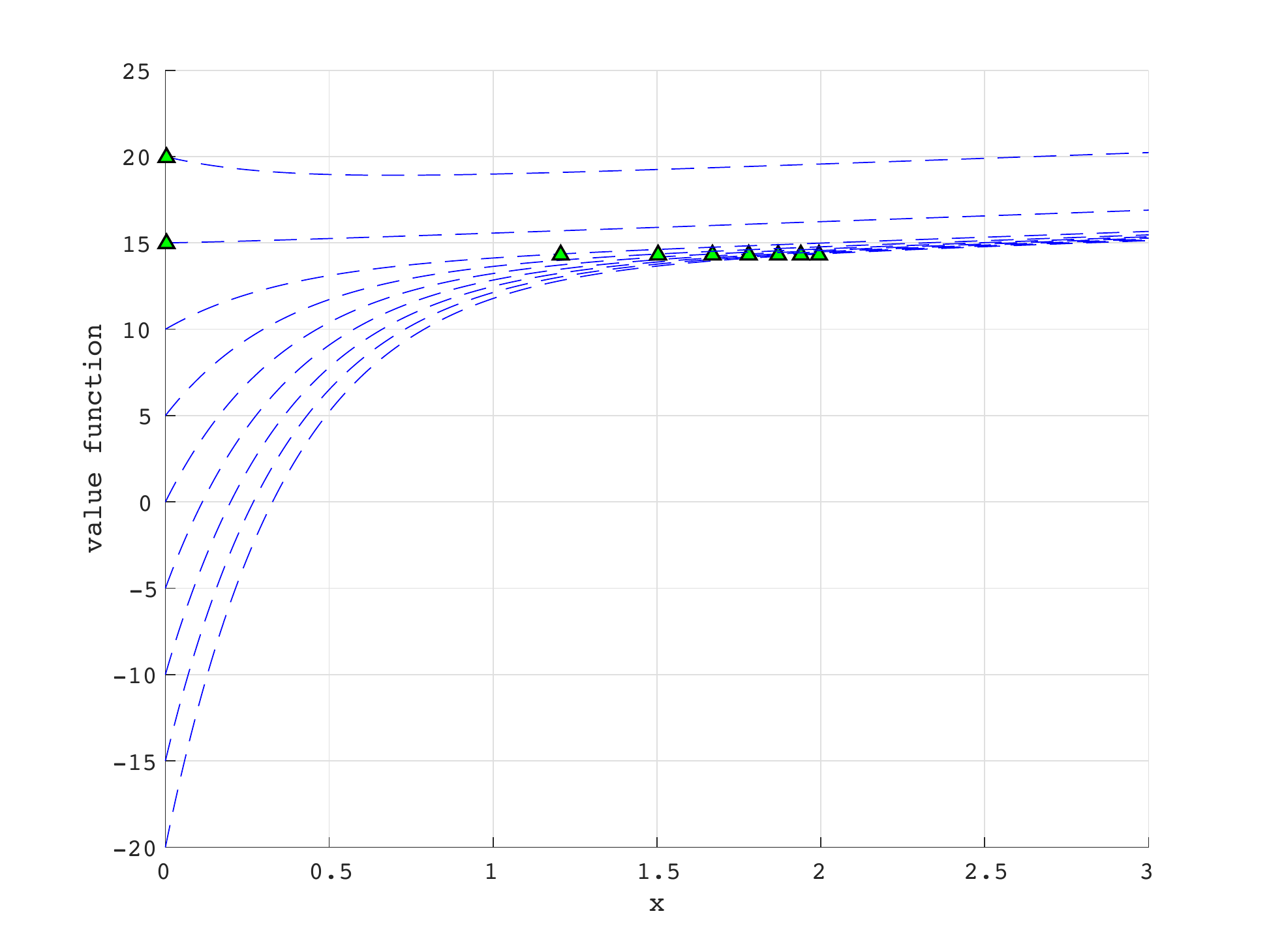}  \\
(i) with respect to $r$ & (ii) with respect to $\rho$  \end{tabular}
\end{minipage}
\caption{(i) The value functions $v_{b^*}$ (dotted) for $r = 0.01$, $0.02$, $0.05$,  $0.1$, $0.2$, $0.5$, $1$, $2$, $3$, $4$, $5$ along with the value function $\tilde{v}$ (given in \eqref{eq:va-star}) in the classical case (solid).  The up-pointing triangles show the points at $b^*$ of $v_{b^*}$; the circle shows the point at $\tilde{b}^*$ of $\tilde{v}$.
(ii) The value functions $v_{b^*}$ for $\rho = -20$, $-15$, \ldots, $15$, $20$.  The up-pointing triangles show the points at $b^*$ of $v_{b^*}$.
} \label{fig_rho}
\end{center}
\end{figure}

\subsection{Numerical results for the second problem}  We now move on to the case with capital injection and confirm the analytical results obtained in Section \ref{section_second_prob}.  Here, we set $\beta = 2$ and use the same parameters as \textbf{Case 1} above, unless stated otherwise.

In Figure \ref{figure_value_function_capital},  we plot the function $\hat{f}$ as well as the value function $u_{b^\dagger}$ along with suboptimal expected NPVs $u_{b}$ for $b \neq b^\dagger$.  Here $\hat{f}$ always starts at a negative value ($1-\beta$) and increases monotonically; its root becomes $b^\dagger$.  The function $u_{b^\dagger}$ is confirmed to dominate $u_b$ uniformly in $x$.

\begin{figure}[htbp]
\begin{center}
\begin{minipage}{1.0\textwidth}
\centering
\begin{tabular}{cc}
 \includegraphics[scale=0.35]{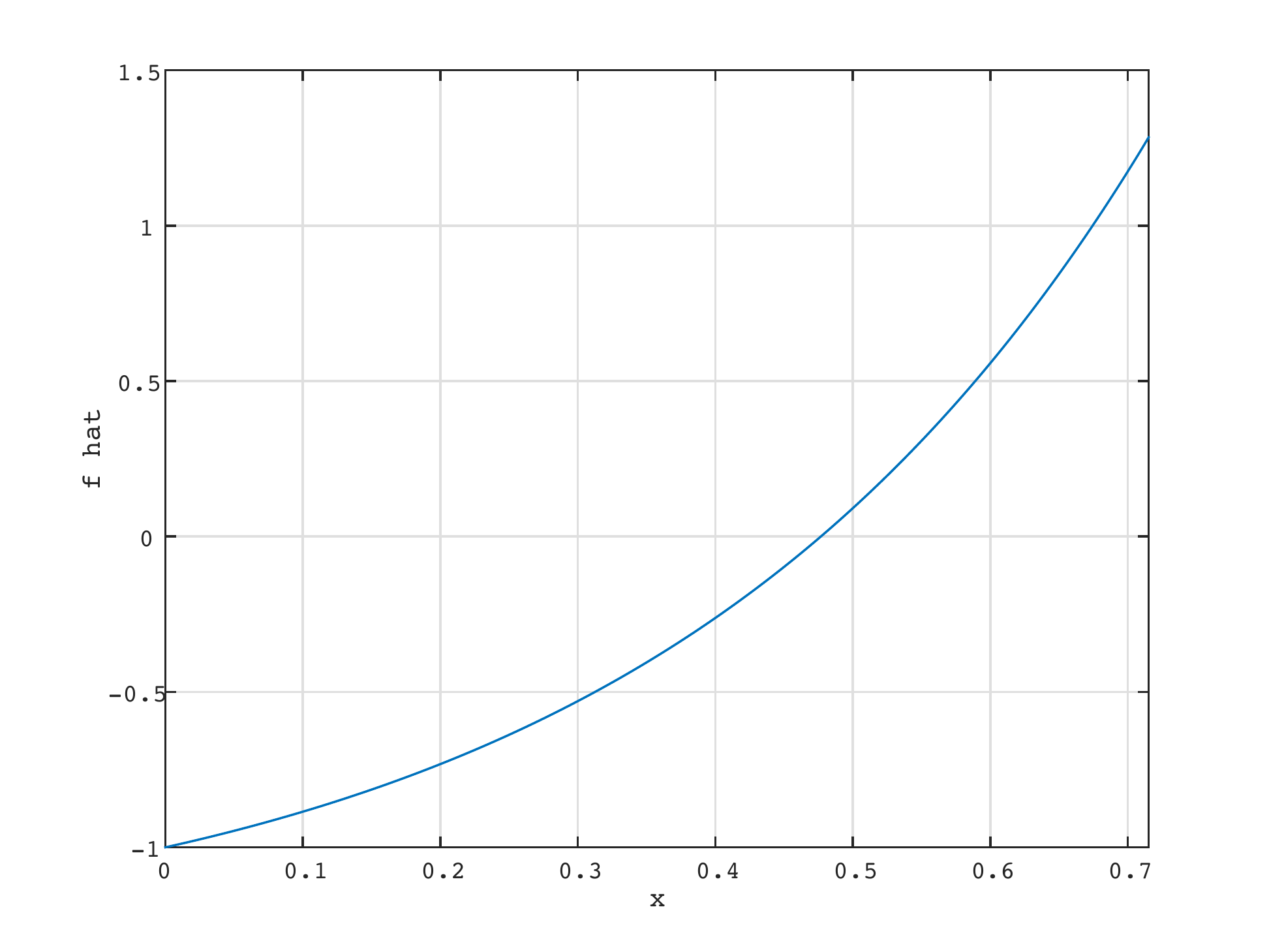} & \includegraphics[scale=0.35]{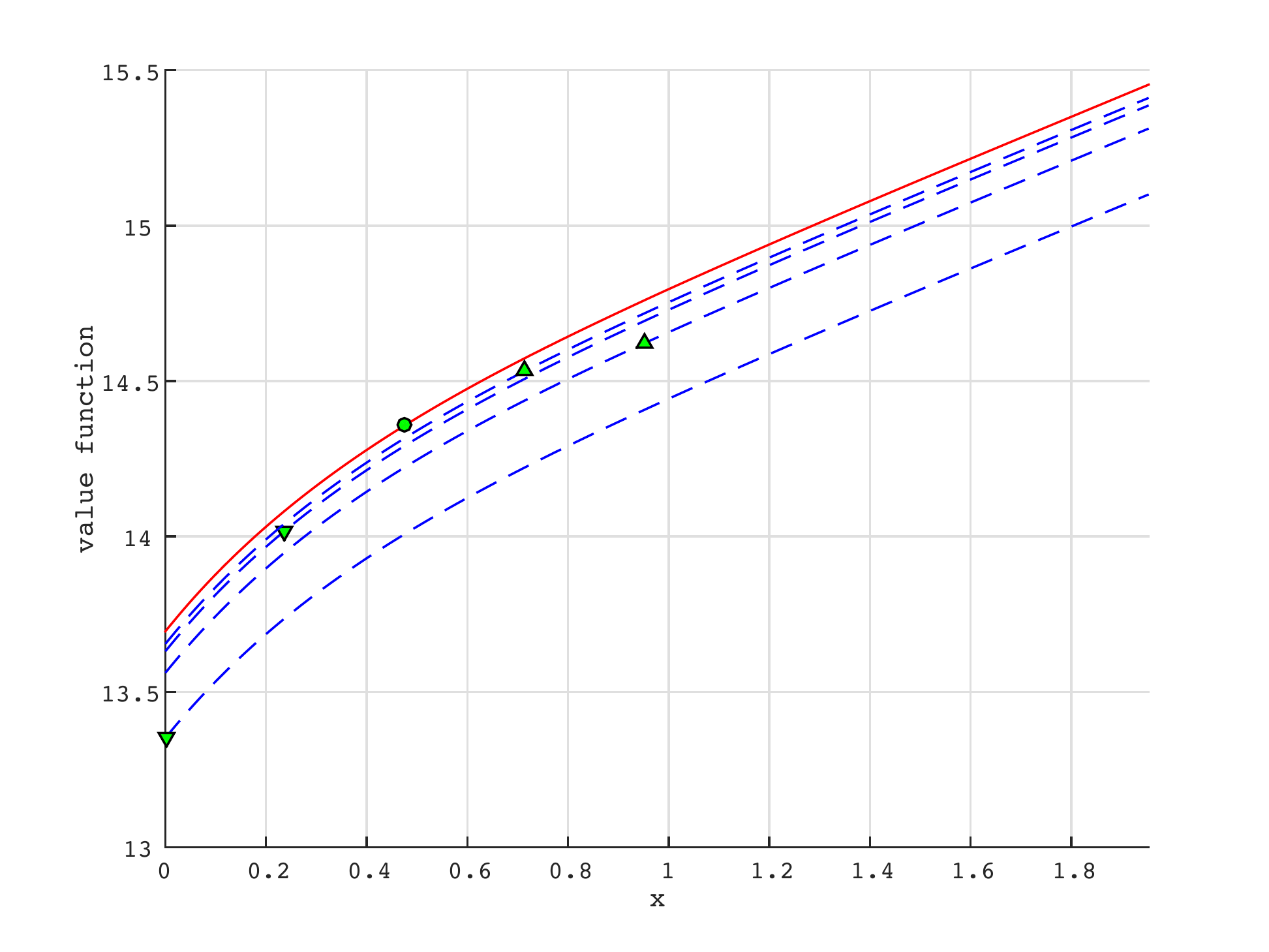}  \\
 $b \mapsto \hat{f}(b)$ & $x \mapsto u_{b^\dagger}(x)$ \end{tabular}
\end{minipage}
\caption{(Left) A plot of $\hat{f}$. (Right) The corresponding value function $u_{b^\dagger}(x)$ (solid) along with suboptimal expected NPVs $u_b$ for $b = 0, b^\dagger/2, 3 b^\dagger/2, 2 b^\dagger$ (dotted).  The values at $b^\dagger$ are indicated by circles whereas those at $b > b^\dagger$ (resp.\ $b < b^\dagger$) for suboptimal expected NPVs are indicated by up-pointing (resp.\ down-pointing) triangles.
}  \label{figure_value_function_capital}
\end{center}
\end{figure}

In Figure \ref{figure_value_function}, we show $u_{b^\dagger}$ for an increasing sequence of $r$ along with those in the classical case $\tilde{u}$; see \eqref{v_bar_a_bailout} and the discussion given in Remark \ref{remark_convergence_r_capital_injection}. 
%\begin{equation} 
%\tilde{u}(x) :=  - \overline{Z}^{(q)}(\tilde{b}^\dagger-x) - \frac {\psi'(0+)} q, \quad x \geq 0,\label{v_bar_a_bailout}
%\end{equation}
%with its optimal barrier $\tilde{b}^\dagger :=  (Z^{(q)})^{-1}(\beta)$
%(see Bayraktar et al.\ \cite{BKY}). 
 Again, the convergence of $u_{b^\dagger}$ and $b^\dagger$ to the classical case $\tilde{u}$ and $\tilde{b}^\dagger$ can be confirmed.

\begin{figure}[htbp]
\begin{center}
\begin{minipage}{1.0\textwidth}
\centering
\begin{tabular}{c}
 \includegraphics[scale=0.35]{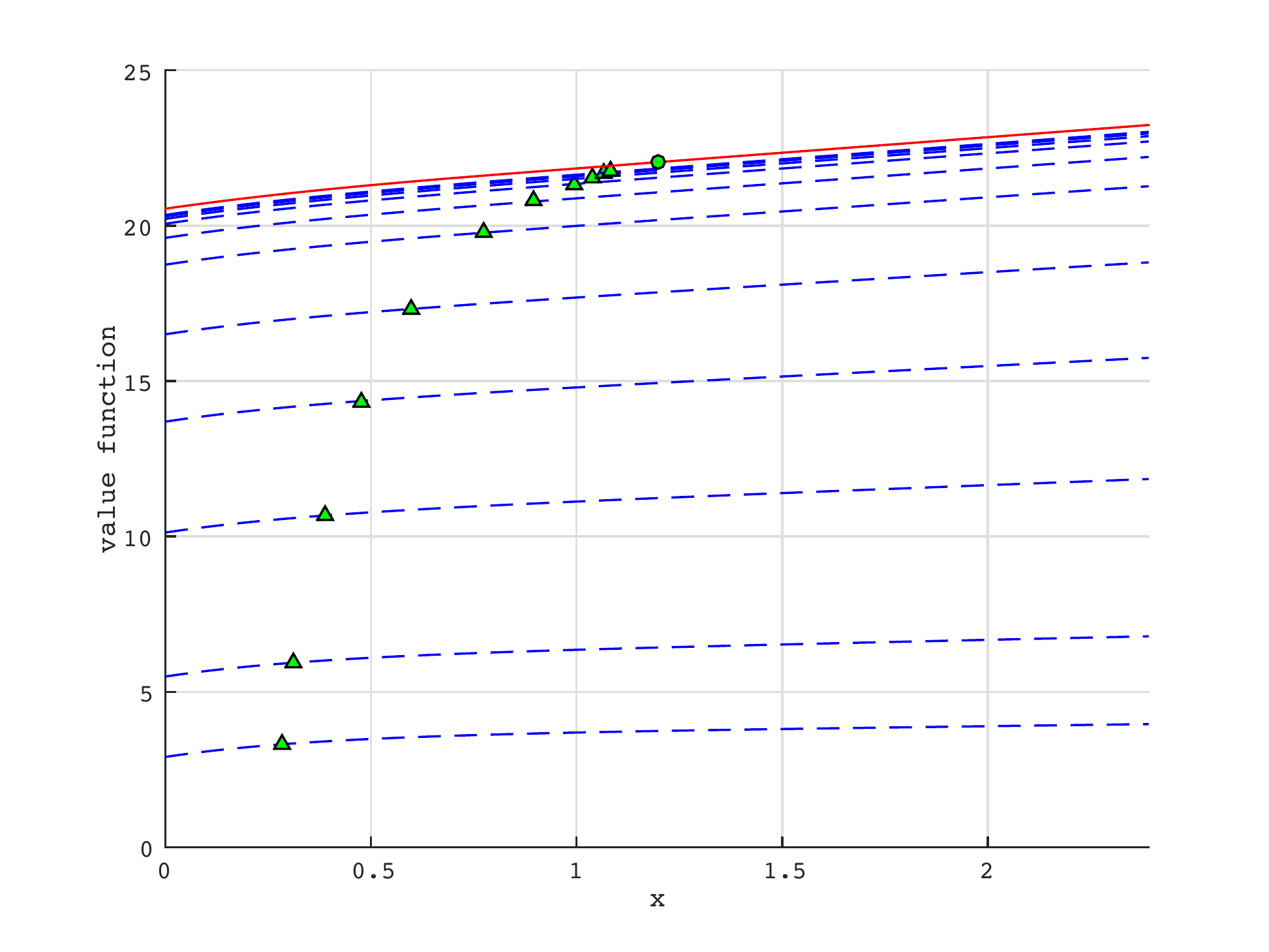}  \end{tabular}
\end{minipage}
\caption{The value functions $u_{b^\dagger}$ (dotted) for $r = 0.01$, $0.02$, $0.05$,  $0.1$, $0.2$, $0.5$, $1$, $2$, $3$, $4$, $5$ along with the value function $\tilde{u}$ (given in \eqref{v_bar_a_bailout}) in the classical case (solid).  The up-pointing triangles show the points at $b^\dagger$ of $u_{b^\dagger}$; the circle shows the point at $\tilde{b}^\dagger$ of $\tilde{u}$.
} \label{figure_value_function}
\end{center}
\end{figure}

\section{Concluding remarks} \label{section_conclusion} %\red{[added this]}

In this paper, we studied optimal dividend problems under the condition that dividend payments can be made only at the jump times of an independent Poisson process.  We showed the optimality of a periodic barrier strategy, which pushes the surplus down to a certain barrier at each dividend payment opportunity.  We also studied the case with additional classical bailouts. For both problems, the solution can be written concisely in terms of the scale function.

We conclude the paper with discussion on two potential extensions of these problems.

First, it is of interest to consider the case where solvency is only monitored periodically (not continuously).  In particular, for the case it is monitored at the arrival times of an independent Poisson process, the ruin time, or the first arrival time at which surplus is negative, corresponds to what is known as a special case of Parisian ruin; this occurs as soon as the surplus process stays below zero consecutively for an exponential time.  It is reasonably conjectured in this case that the periodic barrier strategy is again optimal.

Second, in view of our second problem with classical bail-outs, it makes more sense to consider the case bail-outs are also periodic. 
This can be modeled by introducing jump times of another independent Poisson process, say $M$ with rate $r'$, at which capital injection can be made.  In this case, ruin must be considered and hence the problem is terminated at the first down-crossing time at zero (as opposed to $\infty$).  It is reasonably conjectured that it is optimal to reflect the process in a Parisian fashion from above at $\overline{b}$ and from below at $\underline{b}$ for suitable barriers $0 \leq \underline{b} \leq \overline{b}$. Here $\underline{b} = 0$ can happen as it may be optimal not to inject capital at all when its unit cost is high; in this case, the value function becomes that in our first problem.

Under a suitable condition, we expect, as the frequency of capital injection opportunities $r'$ increases to infinity, the lower barrier $\underline{b}$ decreases to zero and the results converge to the results in our second problem (or otherwise our first problem if it is better not to bail out at all). 
 In view of the numerical results in Figure \ref{fig_rho}, for the dividend maximization in the current paper, the optimal solution converges quickly to the classical case as $r$ increases. Because similar convergence is expected for capital injection, the value function of our second problem can be seen as a rough approximation given that the frequency of capital injection opportunities is high.

\section*{Acknowledgments}
J. L. P\'erez  is  supported  by  CONACYT,  project  no.\ 241195.
K. Yamazaki is in part supported by MEXT KAKENHI grant no.\  26800092 and 17K05377.

\appendix
\section{Proof of Lemma \ref{verificationlemma}}\label{Appendix_A}
By the definition of $v$ as a supremum, it follows that $v_{\hat{\pi}}(x)\leq v(x)$ for all $x\geq0$. We write $w:=v_{\hat{\pi}}$ and show that $w(x)\geq v_\pi(x)$ for all $\pi\in\mathcal{A}$ for all $x\geq0$. 

Fix $\pi\in \mathcal{A}$ and the corresponding surplus process $U^\pi$. 
%We denote $V^{\pi}(t)=X(t)-\sum_{0\leq s\leq t}\nu^{\pi}(s)\Delta N^r(s)$ and 
Let $(T_n)_{n\in\mathbb{N}}$ be the sequence of stopping times defined by $T_n :=\inf\{t>0:U^\pi (t) > n \text{ or }  U^\pi(t)<1/n\}$. %\red{[JL: I guess it should be just $T_n :=\inf\{t>0:{V}^\pi(t)>n \}$; currently $T_n$ does not go to infinity as $n \uparrow \infty$. By assumption, $V^\pi \geq 0$ and so, it should be ok anyways. I don't know if it is necessary but we can extend the definition of $v^b$ on the negative half line; it is differentiable at zero by the remark above.]}
Since $U^\pi$ is a semi-martingale and $w$ is sufficiently smooth on $(0, \infty)$, we can use the change of variables/It\^o's formula %\red{[remove Meyer? Because we have the sufficient smoothness?]} 
(cf.\ Theorems II.31 and II.32 of \cite{protter})  to the stopped process $(\mathrm{e}^{-q(t\wedge T_n)}w(U^\pi(t\wedge T_n)); t \geq 0)$ to deduce under $\mathbb{P}_x$ that%\red{[JL: So $V_{0-} = x$ and so below, we should include the change at $0$ on the right hand side.]}
\begin{equation*}
\label{impulse_verif_1}
\begin{split}
\mathrm{e}^{-q(t\wedge T_n)}w(U^\pi(t\wedge T_n))&-w(x)
=  -\int_{0}^{t\wedge T_n}\mathrm{e}^{-qs} q w(U^\pi(s-)) \mathrm{d}s
+\int_{[0, t\wedge T_n]}\mathrm{e}^{-qs}w'(U^\pi(s-)) \mathrm{d} X(s)  \\
&+ \frac{\sigma^2}{2}\int_0^{t\wedge T_n}\mathrm{e}^{-qs}w''(U^\pi(s-))\mathrm{d}s+\sum_{0 \leq s\leq t\wedge T_n}\mathrm{e}^{-qs}[\Delta w(U^\pi(s-)-\nu^{\pi}(s))\Delta N^{r}(s) ] \\
& + \sum_{0 \leq s\leq t\wedge T_n}\mathrm{e}^{-qs}[\Delta w(U^\pi(s-)+\Delta X(s))-w'(U^\pi(s-))  \Delta X(s)  ].
\end{split}
\end{equation*}
%where we use the following notation: $\Delta \zeta(s):= \zeta(s)-\zeta(s-)$ and $\Delta w(\zeta(s)):=w(\zeta(s))-w(\zeta(s-))$ for any right-continuous process $\zeta$. %\red{[deleted about $N^r$ because it is already defined.]}%and $N^{r}$ represents the independent Poisson process whose arrivals determines the times at which the control acts.
Rewriting the above equation leads to %\red{[changed the arguments below:]}
\begin{equation*}
\begin{split}
\mathrm{e}^{-q(t\wedge T_n)}&w(U^\pi(t\wedge T_n))  -w(x)
\\
= &  \int_{0}^{t\wedge T_n}\mathrm{e}^{-qs}   (\mathcal{L}-q)w(U^\pi(s-))   \mathrm{d}s
-\int_{[0,t\wedge T_n]}\mathrm{e}^{-qs}\nu^{\pi}(s)\mathrm{d}N^{r}(s)  \\
&+\int_0^{t\wedge T_n}\mathrm{e}^{-qs}r\left[\nu^\pi(s)+w(U^\pi(s-)-\nu^{\pi}(s))-w(U^\pi(s-))\right] \mathrm{d}s  + M(t \wedge T_n)
\end{split}
\end{equation*}
where 
\begin{align}\label{def_M_martingale}
\begin{split}
M(t) &:= \int_0^t \sigma  \mathrm{e}^{-qs} w'(U^{\pi}(s-)) \diff B(s) +\lim_{\varepsilon\downarrow 0}\int_{[0,t]} \int_{(\varepsilon,1)}  \mathrm{e}^{-qs}w'(U^{\pi}(s-))y (\mathcal{N}(\diff s\times \diff y)-\Pi(\diff y) \diff s)\\
&+\int_{[0,t]} \int_{(0,\infty)} \mathrm{e}^{-qs} \big[ w(U^\pi(s-)+y)-w(U^\pi(s-))-w'(U^\pi(s-))y\mathbf{1}_{\{y\in (0, 1)\}} \big] (\mathcal{N}(\diff s\times \diff y)-\Pi(\diff y) \diff s)\\
&+\int_{[0,t]} \mathrm{e}^{-qs}\left[\nu^\pi(s)+w(U^\pi(s-)-\nu^{\pi}(s))-w(U^\pi(s-))\right] \mathrm{d} (N^r(s)-rs), \quad t \geq 0, 
\end{split}
\end{align}
where $( B(s); s \geq 0 )$ is a standard Brownian motion and $\mathcal{N}$ %\red{[$N$ is used for Poisson process; change font and use like $\mathcal{N}$?]}
 is a Poisson random measure in   the measure space  $([0,\infty)\times (0, \infty),\B [0,\infty)\times \B (0, \infty), \diff s \times \Pi( \diff x))$.
\par Hence we derive that, by \eqref{surplus_constraint},
\begin{equation*}
\begin{split}
w(x) \geq &
-\int_{0}^{t\wedge T_n}\mathrm{e}^{-qs}  \Big[ (\mathcal{L}-q)w(U^\pi(s-))+r\max_{0\leq l \leq U^\pi(s-)}\left\{ l+w(U^\pi(s-)-l)-w(U^\pi(s-))\right\} \Big]  \mathrm{d}s\\
& + \int_{[0, t\wedge T_n]}\mathrm{e}^{-qs}\nu^{\pi}(s)\mathrm{d}N^{r}(s) - M(t\wedge T_n) + \mathrm{e}^{-q(t\wedge T_n)}w(U^\pi(t\wedge T_n)).
\end{split}
\end{equation*}
Using  the assumption \eqref{HJB-inequality}, we have
\begin{equation} \label{w_lower}
\begin{split}
w(x) \geq &\int_{[0,t\wedge T_n]}\mathrm{e}^{-qs}\nu^{\pi}(s)\mathrm{d}N^{r}(s) + \mathrm{e}^{-q(t\wedge T_n)}w(U^\pi(t\wedge T_n)) - M(t\wedge T_n).
\end{split}
\end{equation} %\red{[JL: Can we just say by compensation formula, $M$ is a martingale?]}
%In addition by the \blue{fact that $w\in C^2(0,\infty)$} \red{and $w(U^\pi(s))$ is bounded a.s.\ on $[0, t \wedge T_n]$?} and by the compensation formula (cf.\ Corollary 4.6 of \cite{K}), the process $(M(t \wedge T_n); t\geq0 )$ is a zero-mean $\mathbb{P}_x$-martingale.   
%\blue{[Kazu: How about this "
In addition by the fact that $w\in C^2(0,\infty)$, and that $U^\pi(s-)$ is bounded a.s.\ on $[0, t \wedge T_n]$, the compensation formula (cf.\ Corollary 4.6 of \cite{K}) implies that the process $(M(t \wedge T_n); t\geq0 )$ is a zero-mean $\mathbb{P}_x$-martingale.

%Indeed, by the L\'evy-It\^o decomposition  the expression in the first curly bracket of \eqref{def_M_martingale} is a zero-mean martingale and by the compensation formula (cf.\ Corollary 4.6 of \cite{K}) the expression in the second curly bracket is also a zero-mean martingale. 	
\par Now  taking expectations  in \eqref{w_lower} and
letting $t$ and $n$ go to infinity ($T_n \xrightarrow{n \uparrow \infty} \tau_0^\pi$ $\mathbb{P}_x$-a.s.), Fatou's lemma (noting that $w$ is bounded from below by $\rho \wedge 0$) and \eqref{v_0_continuity} give  %\red{[with the current definition of $T_n$ it does not go to infinity]}
\begin{equation*}
w(x) \geq \lim_{t, n \uparrow \infty} \mathbb{E}_x \left( \int_{[0, t\wedge T_n]}\mathrm{e}^{-qs}\nu^{\pi}(s)\mathrm{d}N^{r}(s) + \mathrm{e}^{-q(t\wedge T_n)}w(U^\pi(t\wedge T_n)) \right) \geq v_\pi(x).
\end{equation*}
%Hence we proved $w(x)\geq v(x)$ for all $x\geq0$.
This completes the proof.

\section{Proof of Lemma \ref{verificationlemma_2}}\label{Appendix_B}
	%	\red{[JL: So, I am fuzzy about the Ito's formula. But here, we only showed that $v^{b^*}$ is sufficently smooth on $(0, \infty)$. But I guess it needs some smoothness at $\{0\}$ as well? If so, do we need to extend the function to $\R$; the proof still holds by using the Meyer-Ito type, because it is differentiable at $0$ by the Remark below.]}
	
	By the definition of $u$ as a supremum, it follows that $u_{\hat{\pi}}(x)\leq u(x)$ for all $x\geq0$. We write $w:=u_{\hat{\pi}}$ and show that $w(x)\geq u_{\bar{\pi}}(x)$ for all $\bar{\pi}\in\bar{\mathcal{A}}$ and $x\geq0$.  By $\beta > 1$ and the constraint that $U^{\bar{\pi}} \geq 0$, we can focus on $\bar{\pi}$ such that \eqref{surplus_constraint} holds with $\pi$ replaced with $\bar{\pi}$. 
	
	Fix $\bar{\pi}\in \bar{\mathcal{A}}$ and $U^{\bar{\pi}}$ the corresponding surplus process. %We denote $V^{\pi}_t=X(t)-\sum_{0\leq s\leq t}\nu^{\pi}(s)\Delta N^r(s)$ and 
Let $(T_n)_{n\in\mathbb{N}}$ be the sequence of stopping times defined by $T_n :=\inf\{t>0: U^{\bar{\pi}}(t)>n\}$. %\red{[JL: I guess it should be just $T_n :=\inf\{t>0:{V}^\pi(t)>n \}$; currently $T_n$ does not go to infinity as $n \uparrow \infty$. By assumption, $V^\pi \geq 0$ and so, it should be ok anyways. I don't know if it is necessary but we can extend the definition of $v^b$ on the negative half line; it is differentiable at zero by the remark above.]}
	Since $U^{\bar{\pi}}$ is a semi-martingale and $w$ is sufficiently smooth on $(0, \infty)$ and differentiable at zero by assumption, we can use the change of variables/Meyer-It\^o's formula (cf.\ Theorems II.31 and II.32 of \cite{protter})  to the stopped process $(\mathrm{e}^{-q(t\wedge T_n)}w(U^{\bar{\pi}}(t\wedge T_n)); t \geq 0)$ to deduce under $\mathbb{P}_x$ that, with $R^{\bar{\pi},c}$ the continuous part of $R^{\bar{\pi}}$,%\red{[JL: So $V_{0-} = x$ and so below, we should include the change at $0$ on the right hand side.]}
	\begin{equation*}
		\label{impulse_verif_1}
		\begin{split}
			\mathrm{e}^{-q(t\wedge T_n)}w(U^{\bar{\pi}}(t\wedge T_n)&)-w(x)
			=  -\int_{0}^{t\wedge T_n}\mathrm{e}^{-qs} q w(U^{\bar{\pi}}(s-)) \mathrm{d}s
			+\int_{[0, t\wedge T_n]}\mathrm{e}^{-qs}w'(U^{\bar{\pi}}(s-)) \mathrm{d} X(s)  \\
			%&+\int_{[0, t\wedge T_n]}\mathrm{e}^{-qs}w'(U^{\bar{\pi}}(s-)) \mathrm{d} R^{\bar{\pi}}(s)\\
			&+ \int_0^{t\wedge T_n} \mathrm{e}^{-qs}w'(U^{\bar{\pi}}(s-)) \mathrm{d} R^{\bar{\pi},c}(s) + \sum_{0 \leq s\leq t\wedge T_n}\mathrm{e}^{-qs}[\Delta w(U^{\bar{\pi}}(s-)+\Delta R^{\bar{\pi}}(s))] \\
			&+ \frac{\sigma^2}{2}\int_0^{t\wedge T_n}\mathrm{e}^{-qs}w''(U^{\bar{\pi}}(s-))\mathrm{d}s+\sum_{0 \leq s\leq t\wedge T_n}\mathrm{e}^{-qs}[\Delta w(U^{\bar{\pi}}(s-)-\nu^{\bar{\pi}}(s))\Delta N^{r}(s) ] \\
			& + \sum_{0 \leq s\leq t\wedge T_n}\mathrm{e}^{-qs}[\Delta w(U^{\bar{\pi}}(s-)+\Delta X(s))-w'(U^{\bar{\pi}}(s-))  \Delta X(s)  ].
		\end{split}
	\end{equation*}
%	\red{JL: I guess $R$ can have a jump and so $\int_{[0, t\wedge T_n]}\mathrm{e}^{-qs}w'(U^{\bar{\pi}}(s-)) \mathrm{d} R^{\bar{\pi}}(s)$ should be changed to 
%	\begin{align*}
%	\int_{[0, t\wedge T_n]}\mathrm{e}^{-qs}w'(U^{\bar{\pi}}(s-)) \mathrm{d} R^{\bar{\pi},c}(s) + \sum_{0 \leq s\leq t\wedge T_n}\mathrm{e}^{-qs}[\Delta w(U^{\bar{\pi}}(s-)+\Delta R(s)) \Delta R(s)  ].
%	\end{align*}
%where $R^{\bar{\pi},c}$ is the diffusive part of $R^{\bar{\pi}}$.
%	Later we use that $w' \leq \beta$ to show that it is less than or equal to
%	\begin{align*}
%	\int_{[0, t\wedge T_n]}\mathrm{e}^{-qs}\beta \mathrm{d} R^{\bar{\pi}}(s).	\end{align*}
%	}
	%\red{JL: Do we need to add in the above $\sum_{0 \leq s\leq t\wedge T_n}\mathrm{e}^{-qs}[\Delta w(U^{\bar{\pi}}(s-)+\Delta R(s))-w'(U^{\bar{\pi}}(s-))  \Delta R(s)  ]$? I am a little confused, but, if so, we need to assume something more (ok in our case because $w$ is concave, though)? It may be that the problem is assuming that capital injection is not allowed if the surplus is strictly positive?}
	%\red{[remove this redundant sentences that are defined in A]}
	%where we use the following notation: $\Delta \zeta(s):= \zeta(s)-\zeta(s-)$ and $\Delta w(\zeta(s)):=w(\zeta(s))-w(\zeta(s-))$ for any right-continuous process $\zeta$ and $N^{r}$ represents the independent Poisson process whose arrivals determines the times at which the control acts.
	Rewriting the above equation leads to %\red{[changed the arguments below:]}
	\begin{equation*}
		\begin{split}
			\mathrm{e}^{-q(t\wedge T_n)}&w(U^{\bar{\pi}}(t\wedge T_n))  -w(x)
			\\
			= &  \int_{0}^{t\wedge T_n}\mathrm{e}^{-qs}   (\mathcal{L}-q)w(U^{\bar{\pi}}(s-))   \mathrm{d}s
			-\int_{[0,t\wedge T_n]}\mathrm{e}^{-qs}\nu^{\bar{\pi}}(s)\mathrm{d}N^{r}(s)\\%+\int_{[0, t\wedge T_n]}\mathrm{e}^{-qs}w'(U^{\bar{\pi}}(s-)) \mathrm{d} R^{\bar{\pi}}(s)  \\
			&+ \int_0^{t\wedge T_n}\mathrm{e}^{-qs}w'(U^{\bar{\pi}}(s-)) \mathrm{d} R^{\bar{\pi},c}(s) + \sum_{0 \leq s\leq t\wedge T_n}\mathrm{e}^{-qs}[\Delta w(U^{\bar{\pi}}(s-)+\Delta R^{\bar{\pi}}(s))] \\
			&+\int_0^{t\wedge T_n}\mathrm{e}^{-qs}r\left\{\nu^{\bar{\pi}}(s)+w(U^{\bar{\pi}}(s-)-\nu^{\bar{\pi}}(s))-w(U^{\bar{\pi}}(s-))\right\} \mathrm{d}s  + M(t \wedge T_n)
		\end{split}
	\end{equation*}
	where the process $M$ is the martingale \eqref{def_M_martingale} with $\pi$ replaced with $\bar{\pi}$. %\red{[simplified this way, ok?]}
%	\begin{align}\label{def_M_martingale}
%		\begin{split}
%			M(t) &:= \int_0^t \sigma  \mathrm{e}^{-qs} w'(U^{\bar{\pi}}(s-)) \diff B(s) +\lim_{\varepsilon\downarrow 0}\int_{[0,t]} \int_{(\varepsilon,1)}  \mathrm{e}^{-qs}w'(U^{\bar{\pi}}(s-))y (\mathcal{N}(\diff s\times \diff y)-\Pi(\diff y) \diff s)\\
%			&+\int_{[0,t]} \int_{(0,\infty)} \mathrm{e}^{-qs}(w(U^{\bar{\pi}}(s-)+y)-w(U^{\bar{\pi}}(s-))-w'(U^{\bar{\pi}}(s-))y\mathbf{1}_{\{y\in (0, 1)\}})(\mathcal{N}(\diff s\times \diff y)-\Pi(\diff y) \diff s)\\
%			&+\int_{[0, t\wedge T_n]}\mathrm{e}^{-qs}\left( \nu^{\bar{\pi}}(s) +w(U^{\bar{\pi}}(s-)-\nu^{\bar{\pi}}(s))-w(U^{\bar{\pi}}(s-))\right) \mathrm{d} (N^r(s)-rs), \quad t \geq 0, 
%		\end{split}
%	\end{align}
%	where $B$ and $\mathcal{N}$ are defined as in Appendix \ref{Appendix_A}.
	\par On the other hand,  using \eqref{HJB-inequality_db2},
	%the fact that $w'(x)\leq \beta$ for $x>0$, 
	we obtain that
	\begin{align*}
	\int_0^{t\wedge T_n}&\mathrm{e}^{-qs}w'(U^{\bar{\pi}}(s-)) \mathrm{d} R^{\bar{\pi},c}(s)+\sum_{0 \leq s\leq t\wedge T_n}\mathrm{e}^{-qs}[\Delta w(U^{\bar{\pi}}(s-)+\Delta R^{\bar{\pi}}(s))]\\
	&\leq \beta \int_0^{t\wedge T_n} \mathrm{e}^{-qs} \mathrm{d} R^{\bar{\pi},c}(s)+\beta\sum_{0 \leq s\leq t\wedge T_n}  {\rm e}^{-qs}\Delta R^{\bar{\pi}}(s)=\beta\int_{[0, t\wedge T_n]}\mathrm{e}^{-qs} \mathrm{d} R^{\bar{\pi}}(s).
	\end{align*}
%	where $( B(s); s \geq 0 )$ is a standard Brownian motion and $N$ is a Poisson random measure in   the measure space  $([0,\infty)\times (0, \infty),\B [0,\infty)\times \B (0, \infty), \diff s \times \Pi( \diff x))$.
	\par This together with \eqref{surplus_constraint} with $\pi$ replaced with $\bar{\pi}$ gives %Hence we derive that
	\begin{equation*}
		\begin{split}
			&w(x) \geq 
			-\int_{0}^{t\wedge T_n}\mathrm{e}^{-qs}  \Big[ (\mathcal{L}-q)w(U^{\bar{\pi}}(s-))+r\max_{0\leq l \leq U^{\bar{\pi}}(s-)}\left(l+w(U^{\bar{\pi}}(s-)-l)-w(U^{\bar{\pi}}(s-))\right) \Big]  \mathrm{d}s\\
			& +\int_{[0, t\wedge T_n]}\mathrm{e}^{-qs}\nu^{\bar{\pi}}(s)\mathrm{d}N^{r}(s)-\beta\int_{[0, t\wedge T_n]}\mathrm{e}^{-qs} \mathrm{d} R^{\bar{\pi}}(s) - M(t\wedge T_n) + \mathrm{e}^{-q(t\wedge T_n)}w(U^{\bar{\pi}}(t\wedge T_n)).
		\end{split}
	\end{equation*}
	Using  the assumptions \eqref{HJB-inequality_db} and \eqref{HJB-inequality_db3},  we have
	\begin{equation} \label{w_lower_2}
	\begin{split}
	w(x) \geq &
	- M(t\wedge T_n)+ \int_{[0, t\wedge T_n]}\mathrm{e}^{-qs}\nu^{\bar{\pi}}(s)\mathrm{d}N^{r}(s)-\beta\int_{[0, t\wedge T_n]}\mathrm{e}^{-qs} \mathrm{d} R^{\bar{\pi}}(s)-m\mathrm{e}^{-q(t\wedge T_n)}.
	\end{split}
	\end{equation} %\red{[JL: Can we just say by compensation formula, $M$ is a martingale?]}
%	In addition by the compensation formula (cf.\ Corollary 4.6 of \cite{K}), the process $(M(t \wedge T_n):t\geq0 )$ is a zero-mean $\mathbb{P}_x$-martingale.   %Indeed, by the L\'evy-It\^o decomposition  the expression in the first curly bracket of \eqref{def_M_martingale} is a zero-mean martingale and by the compensation formula (cf.\ Corollary 4.6 of \cite{K}) the expression in the second curly bracket is also a zero-mean martingale. 	
	\par Now  taking expectations  in \eqref{w_lower_2} (recall that $M(\cdot \wedge T_n)$ is a zero-mean martingale) and
	letting $t$ and $n$ go to infinity ($T_n\nearrow\infty$ $\mathbb{P}_x$-a.s.), the assumption \eqref{admissibility2} and the monotone convergence theorem give  %\red{[with the current definition of $T_n$ it does not go to infinity]}
	\begin{equation*}
		w(x) \geq \mathbb{E}_x \left( \int_{[0,t\wedge T_n]}\mathrm{e}^{-qs}\nu^{\bar{\pi}}(s)\mathrm{d}N^{r}(s)-\beta\int_{[0, t\wedge T_n]}\mathrm{e}^{-qs} \mathrm{d} R^{\bar{\pi}}(s) \right) =u_{\bar{\pi}}(x).
	\end{equation*}
	%Hence we proved $w(x)\geq v(x)$ for all $x\geq0$.
	This completes the proof.
	
	%\red{JL: Avram Palmowski and Pistorius are saying \eqref{admissibility2} is enough to use the monotone convergence. But I don't understand, I guess the condition \eqref{admissibility2} is strengthened and it is finite in mean?}
	%{\color{blue} Kazu: How about this:
	%\begin{align*}
	%	w(x) &\geq \mathbb{E}_x \left( \int_{0}^{t\wedge T_n}\mathrm{e}^{-qs}\nu^{\bar{\pi}}(s)\mathrm{d}N^{r}(s)-\beta\int_{[0, t\wedge T_n]}\mathrm{e}^{-qs} \mathrm{d} R^{\bar{\pi}}(s) \right) \\
	%	&\geq \mathbb{E}_x \left( \int_{0}^{t\wedge T_n}\mathrm{e}^{-qs}\nu^{\bar{\pi}}(s)\mathrm{d}N^{r}(s)-\beta\int_{[0,\infty)}\mathrm{e}^{-qs} \mathrm{d} R^{\bar{\pi}}(s) \right) 
%		=v_{\bar{\pi}}(x).
%	\end{align*}	
%	And apply monotone convergence in the last equality?
	%} %\red{JL:$ \mathbb{E}_x \left( \int_{0}^{t\wedge T_n}\mathrm{e}^{-qs}\nu^{\bar{\pi}}(s)\mathrm{d}N^{r}(s)-\beta\int_{[0,\infty)}\mathrm{e}^{-qs} \mathrm{d} R^{\bar{\pi}}(s) \right)$ can be $-\infty$ and so I guess monotone convergence cannot be used. For now I think it is safe to assume that the mean is finite?}

	\end{document}